\documentclass[12pt]{amsart}

\usepackage{amsfonts} \usepackage[all,cmtip]{xy} \usepackage{amsmath} \usepackage{amsthm} \usepackage{amssymb}

\usepackage{anysize} \marginsize{1in}{1in}{1in}{1in}

\usepackage{textcomp}

\usepackage{color}

\definecolor{dkgr}{rgb}{1.0, 0.0, 0.0}

\usepackage[bookmarks=true, bookmarksopen=true,%
bookmarksdepth=3,bookmarksopenlevel=2,%
colorlinks=true,%
linkcolor=dkgr,%
citecolor=dkgr,%
filecolor=blue,%
menucolor=blue,%
urlcolor=blue]{hyperref}

\definecolor{prpl}{rgb}{0.7, 0.0, 0.7}

\newtheorem{thm}{Theorem}[section]
\newtheorem{theorem}[thm]{Theorem}
\newtheorem{conj}[thm]{Conjecture}

\newtheorem{proposition}[thm]{Proposition}
\newtheorem{definition}[thm]{Definition}
\newtheorem{notation}[thm]{Notation}
 \newtheorem*{definition*}{Definition}
 \newtheorem{lemma}[thm]{Lemma}

\newtheorem{cor}[thm]{Corollary}
\newtheorem{corollary}[thm]{Corollary} 

\theoremstyle{remark} \newtheorem*{remark}{Remark}

\DeclareMathOperator{\PGL}{PGL}
\DeclareMathOperator{\GL}{GL}

\DeclareMathOperator{\Bun}{Bun}
\DeclareMathOperator{\Pic}{Pic}
\DeclareMathOperator{\Aut}{Aut}
\DeclareMathOperator{\Res}{Res}

\newcommand{\Qlbar}{{\overline{\mathbb{Q}}_{\ell}}}

\newcommand*{\Q}{\mathbb{Q}}

\newcommand*{\Z}{\mathbb{Z}}
\newcommand*{\G}{\mathbb{G}}
 \newcommand*{\A}{\mathbb{A}} 
\newcommand*{\D}{\mathbb{D}} \newcommand*{\N}{\mathbb{N}}

\newcommand*{\C}{\mathbb{C}} \newcommand*{\Hh}{\mathcal{H}} \newcommand*{\p}{\mathbb{P}}
\newcommand*{\PP}{\mathbb{P}} \newcommand*{\Oo}{\mathcal{O}} \newcommand*{\oO}{\mathcal{O}}

\newcommand*{\Ll}{\mathcal{L}}
\newcommand*{\Mm}{\mathcal{M}}

\newcommand*{\F}{\mathbb{F}}

\newcommand{\HC}{\mathrm{H}} \newcommand{\trf}{\mathrm{Tr}_F}

\author{Vivek Shende} \author{Jacob Tsimerman}
\begin{document} \title{Equidistribution in
$\mathrm{Bun}_2(\PP^1)$} \maketitle

\begin{abstract}
Fix a finite field.
A hyperelliptic curve determines a measure on the discrete space of rank two bundles on the projective line: the mass of
a given vector bundle is the number of line bundles whose pushforward it is.  In a sequence of hyperelliptic curves whose genera tend to
infinity, these measures tend to the natural measure on the space of rank two bundles.
This is a function field analogue of Duke's theorem
on the equidistribution of Heegner points, and can be proven similarly: it follows from a manipulation
of zeta functions, plus the Riemann Hypothesis for curves.

Likewise, a sequence of hyperelliptic curves equipped with line bundles gives rise to a
sequence of measures on the space of pairs of rank 2 bundles.
We give a conjectural classification of the possible limit measures which arise; this is a function
field analogue of the ``Mixing Conjecture'' of Michel and Venkatesh.  As in the number field
setting \cite{EMiV}, ergodic theory suffices when the line bundle is sufficiently special.

For the remaining bundles, we turn to geometry
and count points on intersections of translates of loci of special divisors in the Jacobian of a hyperelliptic curve.
To prove equidistribution, we would require two results.  The first, we prove: the upper cohomologies of these loci
agree with the cohomology of the Jacobian.  The second, which we establish in characteristic zero and conjecture
in characteristic $p$, is that the sum of the Betti numbers of these spaces grows at most as the exponential of the genus of
the hyperelliptic curve.
 \end{abstract}

\section{Introduction}

Let $k$ be a number field (e.g. $\Q$) or a function field (e.g. $\F_q(t)$).  Let $G$ be an algebraic
group over $k$, let $\A_k$ be the adeles over $k$, and let $X_G := G(k) \backslash G(\A_k)$ denote the symmetric space.
Since $G(\A_k)$ is locally compact and $G(k)$ is a discrete subgroup, there is a natural $G(\A_k)$ invariant measure $\mu_G$ on $X_G$.
We assume $X_G$ has finite volume and normalize so $\mu_G(X_G) = 1$.
Let $H \subset G$ be a subgroup such that $X_H$ has finite volume, which we normalize to $1$.  For $g \in G(\A_k)$,
there is a map $\rho_{g}: X_H \to X_G$ given by $h \mapsto \rho(h)g$.  Since we assume $X_H$ has finite volume,
there is a pushforward measure $\rho_{g*} \mu_H$.  Many examples in the literature point to the following commonly believed

\vspace{2mm} \noindent {\bf Equidistribution conjecture:}
If $G$ is a connected, simply connected group, the set of measures  $\{ \rho_{g*} \mu_H\} \cup \{0\}$\footnote{
We must include the zero measure since $X_G$ may be noncompact, $H$ may be the identity and $g$ may eventually escape
every compact set.}
 is weak$-*$ closed.
\vspace{2mm}

For non-simply connected groups,
one must be careful to distinguish `connected components' of the symmetric space. For example, if $G=\PGL_2$ then
there is a natural map $\mathrm{det} : X_G \to \mathrm{Cl}(k)/2\mathrm{Cl}(k)$, whose fibers
should be thought of as the `connected components'
of $X_G$. 
Thus for a connected group $G$ with simply connected cover $\widetilde{G}$, we set
$X_{G}^0:= \mathrm{im}(X_{\widetilde{G}}\to X_G)$.  For $g, H$ as above we define 
$$\rho_{g*}^0\mu_H = \frac{{\rho_{g*}\mu_H}_{\mid X_G^0}}{\mu_G(X_G^0)}$$

\vspace{2mm} \noindent {\bf Equidistribution conjecture':}
If $G$ is a connected group with simply connected cover $\tilde{G}$, the set of measures  $\{ \rho_{g*}^0 \mu_H\} \cup \{0\}$
 is weak$-*$ closed.
\vspace{2mm}

The standard approach to this question proceeds via dynamics and measure theory.  Equidistribution statements proven
in this manner include the following.  Over number fields, Ratner's theorem \cite{R1, R2} implies the restriction of the conjecture
to the set of groups $H$ such that for some completion $k_v$ of $k$,  $H_{/k_v}$ is  generated by unipotent elements. In particular,
 this includes all semisimple groups, and in that case there is an effective statement \cite{EMV}.  Over function fields, the analogue
 of Ratner's theorem is not known; on the other hand, the statement is known when $H$ is restricted to the set of semisimple groups \cite{EG}.
The case of $H$ a torus appears to be harder.  Linnik, assuming the general Riemann hypothesis, treated what is essentially the case of $G = \mathrm{PGL}_2$ and $H$ a torus
\cite{Lin}.  A more modern treatment, which moreover establishes some partial results for maximal rank tori in semisimple groups, can be found in \cite{ELMV1, ELMV2, ELMV3}.

Another approach exploits harmonic analysis on symmetric spaces.  Indeed, automorphic forms $\phi$ give a basis for the space of functions on $X_G$, and in some special
cases, there is a period formula expressing $\int \rho_{g}^* \phi\, d \mu_H$ as a special value of a twisted $L$-function associated with $\phi$.  This reduces the problem to establishing
subconvexity for the given $L$-function \cite{Du, W}.  In the function field case, this last step is generally a consequence of Deligne's work on the Weil conjectures \cite{D}.
Duke used this approach to treat the case of tori in $\mathrm{PGL}_2$ \cite{Du}.

Neither approach can currently treat the case when $H$ is permitted to range over low rank tori.  In this paper, we use
geometric methods to study the case where $k = \F_q(t)$, $G = \PGL_2 \times \PGL_2$, $H$ varies over rank 1 tori, and
$g$ is chosen such that $g G(\oO) g^{-1} \cap H$ is a maximal compact in $H$.

We begin by discussing the analogous question for $G = \PGL_2$. We recall
We recall in detail in Appendix \ref{sec:adeles} how to pass to geometry; the result is the following.
A rank 1 non-split torus $H_{/\F_q(t)}$ is canonically associated to a hyperelliptic curve, $\pi: C_H \to \PP^1$.
The only nontrivial maps of symmetric spaces which arise are
\begin{eqnarray*} \pi \circ \otimes L:  \Pic(C)/\pi^* \Pic (\PP^1)  & \to & \Bun_2(\PP^1) \\
M & \mapsto & \pi_* (M \otimes L)
\end{eqnarray*}

In this paper, $\Bun_2$ always refers to $\Bun_{\PGL_2}$. For $G = \PGL_2$, the $\otimes L$ do not affect the pushforward measure, so we suppress them.  
The equidistribution statement in the present case is:

\begin{thm} \label{thm:theta1}
    Let $\pi_i:C_i \to \PP^1$ be a sequence of hyperelliptic curves. Let $\mu_i$ the pushforward of the Haar measure
    on $\mathrm{Pic}(C)/\pi_i^* \mathrm{Pic} (\PP^1)$ to $\mathrm{Bun}_2(\PP^1)$.   If no curve appears infinitely many times, then
    the measures $\mu_i$ converge to the natural measure on $\mathrm{Bun}_2(\PP^1)$.
\end{thm}

Up to normalization, the natural measure on $\mathrm{Bun}_2$ assigns to each point the inverse of the number of automorphisms of the
corresponding vector bundle.
Rank two vector bundles on $\p^1$ are necessarily of the form $\Oo(a) \oplus \Oo(b)$.  Dividing out by line bundles, we take the identification
\begin{eqnarray*}
\mathrm{Bun}_2(\PP^1) & \longleftrightarrow & \N \\
\Oo(a) \oplus \Oo(b) & \longleftrightarrow &  |a-b|
\end{eqnarray*}

It is useful to further separate this according to the parity of $|a-b|$, into
\begin{eqnarray*} \mathrm{Bun}_2^0(\PP^1) & \longleftrightarrow & 2\N \\
 \mathrm{Bun}_2^1(\PP^1) & \longleftrightarrow & 2\N + 1 \end{eqnarray*}

The normalized natural measure on $\mathrm{Bun}_2(\PP^1)$ is characterized by
$$\mu(d + 2 \N) =  \frac{1}{2q^{d-1}} \,\,\,\,\,\,\,\,\,\,\,\,\,\, \mathrm{for }\,\, d > 0. $$


A bundle on $\PP^1$ can be characterized by the cohomology of all of its twists by $\oO(1)$.
In particular, given a hyperelliptic curve, $\pi: C \to \p^1$, (half) the map $\mathrm{Pic}(C)/\mathrm{Pic}(\PP^1) \to \mathrm{Bun}_2(\PP^1)$
is given explicitly by
\begin{eqnarray*}
\phi:\mathrm{Pic}^{g-1}(C)& \rightarrow & 2\N \\ \Ll & \mapsto & 2 \dim \HC^0(C, \Ll) \end{eqnarray*}

As there are only finitely many curves over $\F_q$ of any given genus, the limit in the theorem amounts to a limit as $g \to \infty$,
and, in terms of the explicit formula for the natural measure above, to
the assertion
\begin{equation}
\label{eq:pgl2limit}
\lim\limits_{g\to \infty} \frac{\#_q \{\Ll \in \mathrm{Pic}^{g-1}(C) : \HC^0(C,\Ll) \ge c \}}{\#_q \mathrm{Pic}^{g-1}(C)} = q^{1-2c}
\end{equation}

As $C$ is hyperelliptic, the loci in the numerator can be understood explicitly in terms of the symmetric powers.
An analysis of the zeta function of $C$ yields a proof of Theorem
\ref{thm:theta} along the lines of Duke's argument in the number field setting \cite{Du}.
However, this approach does not extend to the case of rank 1 tori in $G = \PGL_2 \times \PGL_2$.
This is due to the fact that Dukes approach crucially relies on a period integral formula, relating the integral of an automorphic form along a
torus to a special value of some corresponding L-function, and such a formula is absent in this setting.

We describe a different approach.  Because $C$ is hyperelliptic, the locus $\Theta_{g+1-2c}$ in the numerator is isomorphic to
the image under the Abel-Jacobi map of $C^{(g+1-2c)}$, and in particular is of codimension $2c-1$.  We will show it is set theoretically
a complete intersection of ample divisors.
  Were it smooth, we could apply the Lefschetz hyperplane theorem and Poincar\'e duality to compute the higher degree cohomologies
$$H^{> g}(\Theta_{g+1-2c}, \Qlbar) = H^{> g}(\mathrm{Pic}^{g-1}(C), \Qlbar [4c-2](2c-1) )$$
The Grothendieck-Lefschetz trace formula then implies that the LHS and RHS of (\ref{eq:pgl2limit}) differ
by the traces of the lower cohomologies.  Bounding the total dimension of these by
$N^g$ establishes the result for all $q > N$; in this case we may take $N = 4$.

While $\Theta_{g+1-2c}$ is in fact singular, it is nonetheless a homology
manifold \cite{IY}.  Thus we have Poincar\'e duality and may conclude as desired.

\vspace{2mm}

We turn to the case of $G = \PGL_2 \times \PGL_2$.  The maps of symmetric spaces, when
neither projection is trivial, are the following:

\begin{eqnarray*}   \mathrm{Pic}(C)/\pi^* \mathrm{Pic} (\PP^1)  & \to & \mathrm{Bun}_2(\PP^1) \times \mathrm{Bun}_2(\PP^1) \\
M & \mapsto & ( \pi_* (M \otimes L), \pi_* (M \otimes L') )
\end{eqnarray*}

Evidently the pushforward measure only depends on the ratio $L^{-1} L'$, so we henceforth take $L' = 1$.

The equidistribution statement in this case is:

\begin{conj} \label{conj:thetatheta}
Let $(C_i,L_i)$ be a sequence of hyperelliptic curves and line bundles on them.
Assume that for each $N\in\N$, there exists some $A(N)$ such that for $i>A(N)$, $L_i \notin \Theta_{N}$.   Then
some subsequence of the pushforward
measures converge to the natural measure on one of \begin{enumerate}
\item[(0)]  $\mathrm{Bun}_2^0(\PP^1) \times \mathrm{Bun}_2^0(\PP^1) \coprod \mathrm{Bun}_2^1(\PP^1) \times \mathrm{Bun}_2^1(\PP^1)$
\item[(1)]  $\mathrm{Bun}_2^0(\PP^1) \times \mathrm{Bun}_2^1(\PP^1) \coprod \mathrm{Bun}_2^1(\PP^1) \times \mathrm{Bun}_2^0(\PP^1)$
\end{enumerate}

If on the other hand such $A(N)$ do not exist, then there exists an effective divisor $D$ on $\PP^1$ and an infinite subsequence such that
$L_i \cong \oO_{C_i}(D_i)$ and $\pi_*(D_i)=D$. In this case, the pushforward measures for this subsequence converge to $\mu_D$ defined in Appendix \ref{smallshift}.
\end{conj}

Ergodic techniques suffice when the $A(N)$ grow slowly to infinity compared with the genus of the curve, or don't exist.
To establish the statement it remains to study the case $L_i \notin \Theta_{\alpha g(C_i)}$ for any constant $1>\alpha>0$.
We turn to geometry.

The assertion of Conjecture \ref{conj:thetatheta} unpacks as before to statements of the form
\[
\lim\limits_{g\to \infty}
\frac{\#_q (L \Theta_{g-1-2c} \cap \Theta_{g-1-2d}) }{\#_q
\mathrm{Pic}^{g-1}(C)} = q^{-2-2c-2d}  \]
To establish these it would suffice to equate the higher cohomology groups and bound the lower ones. We can accomplish
the first:

\begin{thm} \label{thm:comp}
Fix $L \in J^{2g-a-b}$ and $r > a + b$.  $L \notin \Theta_r$, we have
$$\HC^i(\Theta_{g-a} \cap L-\Theta_{g-b}, \Qlbar) \cong \HC^{i+2a + 2b}(J,\Qlbar)(a+b) \,\,\,\,\,\,\,\,\, \mathrm{for}\,\,\,\,  i > 2g - 2a - 2b - \frac{r - a - b}{2}$$
\end{thm}

Indeed, while these intersections fail to be homology manifolds, we can confine
the failure to high enough codimension to prove the above result.
To conclude a comparison on point counts, we require a bound on the lower cohomologies.

\begin{conj} \label{ibound} There exists a universal constant $N$ such that for  $g \gg c, d$,
we have the bound $$\dim \HC^*(L \Theta_{g-1-2c} \cap \Theta_{g-1-2d}) < N^g$$ \end{conj}

Note that the degree of $\Theta_{g-1}$ is $g!$.  Thus the Bombieri-Katz bound
$\dim \HC^*(\Theta_{g-1}) \ll  (g!)^{g-1}$ is not even close.   Nonetheless,
we can prove the analogous result in characteristic zero.

\begin{thm} \label{thm:bound} Conjecture \ref{ibound} holds in characteristic zero. \end{thm}

Our approach to bound the cohomology rests on the theory of higher discriminants \cite{MS2} and the relation between
vanishing cycles and polar varieties \cite{LT, Ma}.  The key tool in these works is integration with respect to
Euler characteristic in general and  the hyperplane formula for the local Euler obstruction in particular.  Unfortunately
this theory has not been extended to characteristic $p$, and the ignorance of the authors does not allow us to
indicate whether this is a technical or essential limitation.

The organization of this paper is as follows.
In Section \ref{sec:geom} we recall basic facts about the Abel-Jacobi map, noting in particular that for hyperelliptic curves,
the Gauss map  $C^{(d)} \dashrightarrow \mathrm{Gr}(d, T_0 J(C))$  extends to a morphism.
In Section \ref{sec:top}, which is the technical heart of the paper,
we study the cohomology of intersections in the Jacobian of loci of special divisors. This Section contains the proof of
Theorems \ref{thm:comp} and \ref{thm:bound}.
In Section \ref{sec:eq} we prove Theorem \ref{thm:theta1}, and, assuming Conjecture \ref{ibound}, prove \ref{conj:thetatheta}.
Finally,  Appendix \ref{sec:adeles} explains the relation between the adelic and geometric formulations of the equidistribution statements.


\vspace{2mm} \noindent {\bf Acknowledgements}.  We thank Ali Altug, Pierre Deligne, Alexandru Dimca, Steven Kleiman, Robert Lazarsfeld, Luca Migliorini, J\"org Sch\"urmann, and the user
``ulrich'' on mathoverflow for helpful conversations, correspondence, and/or counterexamples.

\section{Geometry of special divisors on hyperelliptic curves} \label{sec:geom}

\subsection{The Abel-Jacobi map}
This subsection contains well known facts about curves, as can be found in \cite[Chapter 1]{ACGH}.

Let $C$ be a smooth curve. We write $C^{(d)} = C^d / \mathfrak{S}_d$ for the $d$'th symmetric product, and $J_d$ for the moduli space of
degree $d$ line bundles on $C$.  Their tangent spaces are given by $T_D C^{(d)} =
\HC^0(D, \oO_D(D))$ and $T_L J = \HC^1(C, \oO_C)$.

\begin{lemma}\label{lem:A} Let $A: C^{(k)}\rightarrow J_k$ be the Abel-Jacobi map $D \mapsto \oO_C(D)$. Then $dA: T_{D} C^{(k)} \to T_{A(D)} J_k$ is the natural map $\HC^0(D, \oO_D(D)) \to
\HC^1(C, \oO_C)$ obtained by taking cohomology of the sequence $$0 \to \oO_C \to \oO_C(D) \to \oO_D(D) \to 0$$ thus its image is the kernel of the surjective map $\HC^1(C, \oO_C) \to
\HC^1(C, \oO_C(D))$ and has dimension $$\dim dA(T_D C^{(k)}) =  g - \dim \HC^1(C,D) = k+1 - \dim \HC^0(C,D).$$
\end{lemma}

For a collection of effective divisors $D_\alpha = \sum_P n_{P,\alpha} \cdot P$ we write \begin{eqnarray*} \bigcap D_\alpha & := & \sum_P (\min_\alpha n_{\alpha,P}) \cdot P \\ \bigcup
D_\alpha & := & \sum_P (\max_\alpha n_{\alpha,P}) \cdot P \end{eqnarray*}

\begin{corollary} \label{cor:A} Consider the addition map $A: C^{(d_1)} \times \cdots \times C^{(d_n)} \to J$. Let $(D_1, \ldots, D_n) \in C^{(d_1)} \times \cdots \times C^{(d_n)}$. Then
\[dA( T_{(D_1, \ldots, D_n)} \prod C^{(d_i)}  ) = dA(T_{\bigcup D_i} C^{(\deg \bigcup D_i)}) \] \end{corollary}

\begin{proposition} \label{prop:finite} Let $p_1, \ldots, p_n$ be distinct points such that $\dim \HC^0(C, \oO(\sum p_i) ) = 1$.  Let $d_1, \ldots, d_n$ be any positive integers.  Then any
nontrivial deformation of $(p_1, \ldots, p_n)$ induces a nontrivial deformation of the line bundle $\oO(\sum d_i p_i)$. \end{proposition} \begin{proof} On $J(C)$, multiplication by $d_i$
scales the tangent space to the identity by $d_i$.  The hypothesis ensures that the tangent direction to deforming the distinct $p_i$ are linearly independent. \end{proof}

\begin{corollary} \label{cor:finite} Let $D_1,\ldots, D_n$ be any divisors such that $\dim \HC^0(C, \bigcup D_i) = 1$, and let $\Ll$ be any line bundle.  Then only finitely many points of
$|\Ll |$ are of the form $\sum n_i D_i$. \end{corollary}

The canonical bundle is base point free: vanishing at
any $p \in C$ is a codimension one condition on sections of $H^0(C, K_C)$, or in other words
containing $p$ is a codimension one on divisors in $|K_C|$.  This gives a map
$C \to \p H^0(C, K_C)^\vee = \p H_1(C, \oO_C) = \p T_0 J$, which is identified
with the differential of the Abel map $\p dA: C \to \p T_0 J$.

\begin{definition}
For $p \in C$, we write $\ell_p:=\p dA(p) \in \p T_0 J$.
\end{definition}

More generally, letting $ G(d, \, \cdot \,)$ denote the Grassmannian of $d$ dimensional
subspaces,
for $d \le g$ the map $A: C^{(d)}\rightarrow J$ induces a Gauss map, defined on the locus of divisors
with $|D| = \{D\}$:
\begin{eqnarray*}
G(dA):C^{(d)} & \dashrightarrow  & G(d, H^1(C, \oO_C)) \\
D & \mapsto & [ dA(T_D C^{(d)}) ] 
\end{eqnarray*}

\subsection{Special divisors on hyperelliptic curves}

Let $\pi: C \to \p^1$ be a hyperelliptic curve.  Let $\tau:C \to C$ be the hyperelliptic involution, and let $\kappa = \pi^* \Oo(1)$.

\begin{notation} On a hyperelliptic curve $\pi: C \to \PP^1$, we say a divisor is hyperelliptic if it is the pullback of a divisor on $\PP^1$. For any effective divisor $D$, we write $D^h$
for the maximal effective hyperelliptic divisor such that $D^r:= D - D^h$ is effective. \end{notation}

\begin{lemma}
\label{lem:kappasections}
 For $0 \le h < g$, we have $\HC^0(C, h\kappa) = \HC^0(\p^1, \Oo(h))$.  In other words, for any $D \in |h\kappa|$, we have $D^r = 0$.  Moreover, $(g-1)\kappa$ is the canonical
bundle of $C$. \end{lemma} \begin{proof} Certainly $\C^g = \HC^0(\p^1, \Oo(g-1)) \subset \HC^0(C,(g-1)\kappa)$.  But no degree $2g-2$ line bundle on a curve has more than $g$ sections, so
the inclusion must be an equality.  Moreover the only such bundle is the canonical bundle.  Now we induct downward.  Assuming the statement for $h+1$, consider the sequence $0 \to h\kappa
\to (h+1)\kappa \to \Oo_p \oplus \Oo_{\tau(p)} \to 0$, obtained by pulling back a section of $\Oo(1)$ vanishing away from a ramification point of $\pi$.  Then taking cohomology we have \[ 0
\to \HC^0(C, h\kappa) \to \HC^0(C, (h+1)\kappa) \to \C^2 \to \HC^1(C,h\kappa) \to \HC^1(C,(h+1)\kappa) \to 0\] By hypothesis $\HC^0(C, (h+1)\kappa) = \HC^0(\p^1, \Oo(h+1))$.  Therefore the
image of $\HC^0(C, (h+1)\kappa) \to \C^2$ is the diagonal $\C$, and $\dim \HC^0(C,h\kappa) = h+1$.  As this space contains $\HC^0(\p^1, \Oo(h))$, the containment must be an equality.
\end{proof}

\begin{corollary}
The hyperelliptic involution acts trivially on $\p H^0(C, K_C) = \p T_0 J^\vee$ and consequently
trivially on $\p H^1(C, \oO_C) = \p  T_0 J$.  \end{corollary}

\begin{corollary}
The canonical morphism
$\p(dA): C \to \PP T_0 J$ factors through the hyperelliptic involution, and is in fact
the composition of $C \to C/\tau = \p^1$ with the Veronese embedding
$\p^1 \to \p H^0(\p^1, \oO_{\p^1}(g))^\vee = \p^{g-1}$.
\end{corollary}

\begin{corollary} \label{cor:Gauss}
For $d \le g$,
the Gauss map $G(dA): C^{(d)} \dashrightarrow G(d, T_0 J)$ extends to a morphism
\begin{eqnarray*}
\overline{G}:C^{(d)} & \rightarrow  & G(d, H^1(C, \oO_C)) \\
D & \mapsto &  [ dA(T_{\pi^* \pi_* D} C^{(2d)})]
\end{eqnarray*}
\end{corollary}
\begin{proof}
By Lemma \ref{lem:kappasections}, we have $\dim \HC^0(C, \oO_C(\pi^* \pi_* D)) = d + 1$.
Thus from the long exact sequence
$$ 0 \to \HC^0(C, \oO_C) \to \HC^0(C, \oO_C(D + \overline{D}))
\to \HC^0(C, \oO_{D + \overline{D}}(D + \overline{D})) \xrightarrow{dA} \HC^1(C,\oO_C) $$
we see that $\dim dA(T_{\pi^* \pi_* D} C^{(2d)}) = d$.
On the other hand
$dA(T_D C^{(d)}) \subset dA(T_{\pi^* \pi_* D} C^{(2d)})$,
and thus we have equality where the Gauss map was originally defined.
\end{proof}

The map $\overline{G}$ is evidently invariant under the hyperelliptic involution, and so descends
to $\overline{G}/\tau :\PP^d\rightarrow G(d,\HC^1(C,\oO_C))$.

\begin{lemma}\label{embedding}
For $d < g$, the map $\overline{G}/\tau :\PP^d\rightarrow G(d, T_0 J)$ is an embedding.
In particular, when $d = g-1$ we find an isomorphism
$\overline{G}/\tau: (C/\tau)^{(g-1)} \cong \PP T_0 J ^\vee$.
\end{lemma}
\begin{proof}
As we have seen, $\p dA/\tau = \overline{G}/\tau: (C/\tau) \to \PP T_0 J$ identifies
  $(C/\tau) = \PP^1$ as the rational normal curve in $ \PP T_0 J$.  More generally,
the map $\overline{G}/\tau$ takes $d$ points on the rational normal curve to the
$d-1$-dimensional plane passing through them; there always exists a unique such
plane by the non-degeneracy of the Vandermonde determinant.
\end{proof}

\begin{corollary} \label{tangents} For $k \le g$ and $p_1, \ldots, p_k \in C$,
the dimension of the linear subspace of $\p T_0 J$ spanned by
$\{\ell_{p_1}, \ldots, \ell_{p_k} \}$ is equal to
 $\# \{\pi(p_1), \ldots, \pi(p_k)\} - 1$.
 \end{corollary}

\begin{lemma}\label{hyperpair} Let $D$ be an effective divisor.  Then the following are equivalent: \begin{itemize}
 \item $D$ is special, i.e. $\mathrm{h}^1(D) > 0$.
 \item $\deg D^h/2 + \deg D^r \le g-1$
\end{itemize} Moreover, if $\deg D^h/2 + \deg D^r \le g$, then
the inclusion $|D| \supset D^r + |D^h| = D^r + |(\deg D^h/2 )\kappa|$ is an equality. \end{lemma} \begin{proof} As $\HC^1(C,D)=\HC^0(C,K-D)$, there is
some effective divisor $E$ with $D+E = K$. Since $\HC^0(C,K)\cong\HC^0(\p^1,\Oo(g-1))$, we have $(D + E)^r = 0$ and thus $D^r = \tau(E^r)$.  In any case $D^r + \tau(D^r) + D^h + E^h = K$, so
$\deg D^r + \deg D^h / 2 \le g-1$. Conversely in this case, $D + \tau(D^r)$ is a section $n \kappa$ for $n = \deg D^r + \deg D^h / 2 \le g-1$, so we may find an effective divisor $E \in
\tau(D^r) + |(g-1-n)\kappa|$ such that $D + E = K$.  Finally if $\widetilde{D}$ is linearly equivalent to $D$, then since $\widetilde{D}+E = K$, we must have $\widetilde{D}^r = \tau(E^r) =
\tau(\tau(D^r)) = D^r$.

It remains to treat the case $\deg D^h/2 + \deg D^r = g$.  In this case, by what was just proven, $D$ is not special and the result follows from Riemann-Roch.
 \end{proof}

\begin{remark} The general divisor $D$ of
degree $g+1$ has $\HC^0(C,D) = 2$ but $\HC^0(C, \kappa^{-1}(D)) = 0$. \end{remark}

%
%
%
%
%

\begin{notation}
Twisting by $\kappa$, we identify $\otimes \kappa: J_d(C) \equiv J_{d+2}(C)$.
We denote by $\Theta_d$ the image of $C^{(d)}$.  By Lemma \ref{hyperpair},
we have
\[ \Theta_d(C) = \{\Ll \in J_{d+2r}(C)\,|\, \dim \HC^0(C,\Ll) > r \} \,\,\,\,\,\,\,\,\,\,\,\,\,\,\,\, \mbox{for $d+2r < g$} \]
\end{notation}

We record the following interesting fact, which is not however used in the paper.

%
%
%

\begin{proposition} \label{nash}
$C^{(d)}$ is the Nash blowup of $\Theta_{d}$.
\end{proposition}
\begin{proof}

Corollary \ref{cor:Gauss} asserts that the map $A \times \overline{G}: C^{(d)} \to J \times G(d, T_0 J)$
maps $C^{(d)}$ to the Nash blowup of $\Theta_{d}$.  According to Lemma \ref{embedding}, this
factors as
$$C^{(d)} \xrightarrow{A \times \pi^{(d)}} J \times (C/\tau)^{(d)} \hookrightarrow J \times G(d, T_0 J).$$

Thus it suffices to show $A \times \pi^{(d)}$ is an embedding.
We first show it separates points.  By Lemma \ref{hyperpair},
we have $A(D_1) = A(D_2) \iff D_1 \in |D_2| \iff D_1^{r} = D_2^{r}$. If in addition
we had $\pi^{(d)}(D_1) = \pi^{(d)}(D_2)$ we would necessarily have
$D_1^h = D_2^h$ and hence $D_1 = D_2$.

Next, we show that $\phi$ separates tangent vectors. The derivative of the maps are
given by  $$dA: \HC^0(D,\oO_D(D))\rightarrow \HC^1(C,\oO_C)$$ and
$$ d \pi^{(d)}: \HC^0(D,\oO_D(D))\rightarrow \HC^0(\pi_*D,\oO_{\pi_*D}(\pi_*D )),$$

An element of $\mathrm{ker}(dA)$ must be in the image of $H^0(C,\oO_C(D))$; let
$f$ be such an element viewed as a nonconstant meromorphic function with poles
in $D$.  For any
$\omega \in \HC^0(E, \Omega_C|_E)$ we have the nondegenerate residue pairing
$\textrm{Res}( f\pi^*(\omega)) = \mathrm{Res}( (d \pi^{(d)}_*f) \omega)$.  Thus to show
$f\not\in\mathrm{ker}(d \pi^{(d)})$ it suffices to find $\omega \in \HC^0(\pi_*D, \Omega_C|_{\pi_* D})$
such that $\textrm{Res}( f\pi^*(\omega))\neq 0$. Since $\HC^0(C,K_C-D)\neq 0$, by lemma \ref{hyperpair} we know that $f=\bar{f}$ and thus $f=\pi^{-1}(g)$ for some $g\in \HC^0(\PP^1,\oO(E))$.  The existence of the desired $\omega$ now follows from
nondegeneracy of the residue pairing on $\pi_*D$.
\end{proof}

\subsection{Pairs of special divisors} \label{subsec:pairs}

Consider the following commutative diagram.  The maps $\Sigma$ add line bundles or divisors, and
the maps $A$ are the Abel map.  All quadrilaterals which do not contain squiggly edges
are cartesian.

$$
\xymatrix{
(A \times A)^{-1} \Sigma^{-1} (\Ll) \ar[ddd]_{(A\times A)_{\Ll}} \ar@{~>}[rrr] \ar[rd] & & &  |\Ll| \ar[ld] \ar[ddd] \\ 
& C^{(g-a)}\times C^{(g-b)} \ar@{~>}[r]^{\,\,\,\,\,\,\,\,\,\,\,\,\,\,\Sigma}  \ar[d]_{A \times A}
\ar@/_1pc/@{~>}[rd]^{\underline{A}} &C^{(2g-a-b)}\ar[d]^{A} &\\
& \Theta_{g-a} \times\Theta_{g-b} \ar[r]_{\Sigma} &  J_{2g-a-b} & \\
\Theta_{g-a} \cap \Ll - \Theta_{g-b} \ar[rrr] \ar[ru]
& & &  \{\Ll\}  \ar[ul]
}
$$


We are  interested in the space $\Theta_{g-a} \cap \Ll - \Theta_{g-b}$.
Unfortunately, $(A
\times A)^{-1} \circ s^{-1}(\Ll)$ need not be nonsingular.  We characterize its singular locus:

\begin{proposition} \label{prop:tangents} Consider $\underline{A} = s \circ (A \times A): C^{(g-a)} \times C^{(g-b)} \to J$.
Then $(D_1, D_2)$ is a singular point of the fibre
containing it iff $D_1 \cup D_2$ is special.  More generally, \[d\underline{A}(T_{(D_1, D_2)} C^{(g-a)} \times C^{(g-b)}) = dA(T_{D_1 \cup D_2} C^{(\deg D_1 \cup D_2)})\] has dimension
$\min(g, \deg (D_1 \cup D_2)^h /2 + \deg (D_1 \cup D_2)^r)$
\end{proposition} \begin{proof} The equality follows from Lemma \ref{lem:A} and
Corollary \ref{cor:A}.  According to Lemma \ref{lem:A}, either the tangent map is surjective, or $D_1 \cap D_2$ is special, and in the latter case, according to Lemmas \ref{lem:A} and \ref{hyperpair}, has image of the
stated dimension. \end{proof}

In order to understand $(D_1 \cup D_2)^r$ and $(D_1 \cup D_2)^h$, we will keep track of shared points between $D_1$ and $D_2$, and also of hyperelliptic pairs.

\begin{definition} Given a pair $(D_1, D_2)$ of hyperelliptic divisors, we define its canonical decomposition to be the septuple $(H_\cap,H_1,H_2,S, R_{\cap} ,R_1,R_2)$ determined by the
following formulas: \begin{itemize} \item $H_\cap = D_1^h \cap D_2^h$ \item $H_i = D_i^h - H_\cap$ \item $S = D_1^r \cap ( D_1^r + D_2^r )^h$ \item $R_\cap = (D_1^r  - S) \cap (D_2^r -
\overline{S})$ \item $R_1 = D_1^r - S - R_\cap$ and $R_2 = D_2^r - \overline{S} - R_\cap$ \end{itemize} or equivalently by the following conditions \begin{itemize} \item $H_\cdot$ are hyperelliptic.
\item $R_1 +R_2 + 2R_{\cap}$, $S + R_1 + R_\cap$, and $\overline{S} + R_2 + R_\cap$ each contain no hyperelliptic divisor. \item $H_1 \cap H_2 = \emptyset = R_1 \cap R_2$ \item
$D_1=H_\cap+H_1+S+R_{\cap}+R_1$ \item $D_2=H_\cap+H_2+\overline{S}+R_{\cap}+R_2$ \end{itemize} \end{definition}

Note that \begin{eqnarray*} (D_1 + D_2)^r &  = & 2R_{\cap} + R_1+R_2\\ (D_1 + D_2)^h & = & 2 H_\cap + H_1 + H_2 + S + \overline{S} \end{eqnarray*}

\begin{corollary} \label{cor:tdim} The dimension of the image of $d \underline{A}$ at $(D_1, D_2)$ can be expressed in terms of the canonical decomposition as \[\min(g,\frac{\deg H_\cap+\deg H_1+\deg H_2}{2}
+\deg S + \deg R_\cap + \deg R_1+ \deg R_2)\] \end{corollary}

\begin{definition}
We write $e(L) := \max \{m\, | \,  \HC^0(C, \kappa^{-m} L) \ne 0 \}$.
\end{definition}

In $J^{2g-a-b}$, we have $e(L) \ge (2g-a-b -1)/2$.  Taking the convention that $\Theta_k = A(C^{(k)})$ even for $k \ge g$,
we have
$\Theta_{k-2E} = \{L \in J^{k}: e(L) \ge E\}$.

\begin{cor} \label{cor:dimestimate}
The irreducible components of the singular locus of $\underline{A}^{-1}(L)$ are enumerated by finitely many expressions of the form
$L = \mathcal{O}(R_1
+ R_2 + 2 R_\cap) \otimes \kappa^{n}$.  Each is of dimension at most $e(L)$.
Moreover, if $e(L) \le g - a - b$, then $\underline{A}^{-1}(L)$ and thus also $\Theta_{g-a} \cap L - \Theta_{g-b}$ have the expected dimension
$g- a - b$.
\end{cor} \begin{proof} Let $(D_1, D_2)$ be a point in the singular locus.  Denote the pieces of the canonical decomposition as above. Then $R_1 + R_2 + R_\cap$ is a divisor of
degree at most $g-1$ containing no hyperelliptic pairs; it follows that $\dim \HC^0(C, \oO(R_1 + R_2 + R_\cap)) = 1$.  A fortiori the same is true of their union.  On the other hand we have
$\oO(R_1 + R_2 + 2 R_\cap) = \kappa^{-m} \Ll$ for some $m \leq e(L)$.  By Corollary \ref{cor:finite}, only finitely many choices of the $R_\cdot$ are possible.  Moreover the possibilities for the
corresponding $H_\cdot, S$ vary in a family of dimension $\deg (H_\cap + H_1 + H_2 + S) / 2 \le (\deg \Ll - \deg  \kappa^{-m} \Ll) / 2 = m$.
\end{proof}

\section{Topology of special divisors on hyperelliptic curves} \label{sec:top}

In this Section we take \'etale cohomology with $\ell$-adic coefficients for some
$\ell$ prime to the characteristic of our base field $k$, as in \cite{SGA4, SGA45, D}.
We also require the theory of perverse sheaves \cite{BBD}.

\subsection{Semismall maps and IC sheaves} \label{subsec:semismall}

\begin{lemma} \label{lem:bm1} \cite{IY}\footnote{The origin of this result is slightly mysterious.  The case $i=g-1$ is used in \cite{N} who alleges that it is proved in \cite{BB}.  However
in \cite{BB} the authors deal exclusively with the {\em non}-hyperelliptic case, where the statement is not true and not claimed.  Indeed we asked Bressler, who confirmed that no such
statement appeared in the paper and moreover claimed not to know a proof of the result.  In \cite{IY}, the result is attributed to \cite{BB} for $i = g-1$ and proven in general.} The
$\Theta_i$ are homology manifolds, i.e.,  $\mathrm{IC}_{\Theta_i} = \Qlbar[\dim \Theta_i]$. \end{lemma} \begin{proof} We recall
from \cite{IY} the proof. One applies the method of Borho and Macpherson \cite{BM}. The Abel-Jacobi map $A: C^{(i)} \mapsto \Theta_i$ has fibres $\p^j$ over  $\Theta_{i-2j}^\circ := \Theta_{i-2j}
\setminus \Theta_{< i-2j}$.  It follows that $A$ is {\em semi-small}, i.e., the locus $\Delta_\delta$ where the fibre dimension is $\ge \delta$ is of codimension $\ge 2\delta$.  This
implies first of all that $A_* \Qlbar[i]$ is perverse, and moreover that the terms which may appear in its Beilinson-Bernstein-Deligne \cite{BBD} decomposition are IC sheaves whose
supports are the components of $\Delta_\delta$ of codimension exactly $2\delta$.  We see that $\Delta_{j} = \Theta_{i-2j}$.  We now argue by induction that the contribution of the stratum
$\Theta_{i-2j}$ is precisely $\mathrm{IC}_{\Theta_{i-2j}}$.  It suffices to check this on $\Theta_{i-2j}^\circ$.  Here,
the perverse sheaf $\pi_* \Qlbar[i]$ restricts to a local system with fibre $\Qlbar[i] \oplus \Qlbar[i-2] \oplus \ldots \Qlbar[i-2j]$; the local system is moreover constant because the cohomology of
$\p^j$ has a canonical generator.  As a point $p \in \Theta_{i-2j}^\circ$ lies in the closure of all $\Theta_{i-2k}^\circ$ for $k < j$, we have $\Qlbar[i] \oplus \Qlbar[i-2] \oplus \ldots \Qlbar[i-2j] =
F[i-2j]_p \oplus \bigoplus_{k < j} \mathrm{IC}_{\Theta_{i-2k}}|_p$, where $F$ is the local system giving the contribution of $\Theta_{i-2j}$.  But since evidently
$\mathcal{H}^{-i+2k}(\mathrm{IC}_{\Theta_{i-2k}}) |_p = \Qlbar$,
all summands are accounted for by these lowest cohomology groups and we moreover have $F = \Qlbar$.  Additionally we see that the IC sheaves have no stalk
cohomology except in minimum degree, and hence are (trivial) local systems. \end{proof}

\begin{lemma} \label{thetaislci}
For each $m\gg 0$, $\Theta_i \subset \Theta_{i+1}$ is set-theoretically
the zero section of an ample line bundle defined over $\F_{q^m}$.
\end{lemma}
\begin{proof}
It is well known that $\Theta_{g-1}$ is an ample divisor. For $m\gg_c 0$ there exists a non-Weierstrass point
$Q \in C(\F_{q^m})$.  For $i< g-1$,
\begin{equation}\label{ThetaLCI}
(\Theta_{g-1}+Q)\cap (\Theta_{i+1}+(g-i-1)\tau(Q))=\Theta_i+\kappa+(g-i-2)\tau(Q).
\end{equation}
Indeed, consider an effective divisor $D_{i+1}$ of degree $i+1$
such that $D_{i+1}+(g-i-1)\tau(Q)$ is linearly equivalent to a divisor containing $Q$. By Lemma $\ref{hyperpair}$ we
know that either $Q\in D_{i+1}$, or $D_{i+1}$ contains a hyperelliptic pair. In either case the conclusion is easily seen to follow.
\end{proof}

From the Lefschetz hyperplane theorem, we deduce:

\begin{cor}\label{stability1} We have $\HC^j(\Theta_i \otimes \overline{F}_q,\Qlbar)\cong \HC^j(J \otimes \overline{F}_q,\Qlbar)$ for all $0\leq j<i$, and the isomorphism preserves the generalized eigenspaces of Frobenius. \end{cor}
\begin{proof}

Lemma \ref{thetaislci} and the Lefschetz hyperplane theorem give the isomorphism.  Moreover,
we know that  for $m\gg 0$, the actions of $\mathrm{Frob}_q^m$ agree.  The fact that the
generalized eigenspaces coincide follows formally.
%
%
\end{proof}

By Lemma \ref{lem:bm1}, we can apply Poincare duality to  conclude

\begin{theorem} \label{thm:thetatop} $\HC^j(\Theta_i, \Qlbar)\cong \HC^{j+2g-2i}(J,\Qlbar)(g-i)$ for all $i<j$.
\end{theorem}

We now study the extent to which this remains true for pairs of special divisors.  We preserve the notation
of Section \ref{subsec:pairs}, in particular the notation for the following maps:

$$
\xymatrix{
(A \times A)^{-1} \Sigma^{-1} (\Ll) \ar[ddd]_{(A\times A)_{\Ll}} \ar@{~>}[rrr] \ar[rd] & & &  |\Ll| \ar[ld] \ar[ddd] \\ 
& C^{(g-a)}\times C^{(g-b)} \ar@{~>}[r]^{\,\,\,\,\,\,\,\,\,\,\,\,\,\,\Sigma}  \ar[d]_{A \times A}
\ar@/_1pc/@{~>}[rd]^{\underline{A}} &C^{(2g-a-b)}\ar[d]^{A} &\\
& \Theta_{g-a} \times\Theta_{g-b} \ar[r]_{\Sigma} &  J_{2g-a-b} & \\
\Theta_{g-a} \cap \Ll - \Theta_{g-b} \ar[rrr] \ar[ru]
& & &  \{\Ll\}  \ar[ul]
}
$$

The intersections remain set theoretically lci.  In fact,

\begin{lemma} \label{thetathetaislci}
Fix $L$ with $e(L)\leq g-a-b$, and view $\Sigma^{-1}(L)$  as a subvariety of $J$ by projection to the first co-ordinate.
For $m\gg0$, we have that $\Sigma^{-1}(L)$
 is set theoretically the intersection of $a+b$ translates of $\Theta_{g-1}$ defined over $\F_{q^m}$.
\end{lemma}

\begin{proof}
By Lemma \ref{thetaislci}, $\Theta_{g-a}$ and $\Theta_{g-b}$ are set theoretically the intersection of $a$ and $b$ translates of $\Theta_{g-1}$ respectively. The lemma then follows from the dimension estimate in Corollary \ref{cor:dimestimate} and the fact that $\Sigma^{-1}(L)=\Theta_{g-a}\cap L-\Theta_{g-b}$.
\end{proof}

\begin{corollary} \label{thetathetalef}
Assume $e(L)\leq g-a-b$.
We have
$$\HC^j(\Theta_{g-a} \cap \Ll - \Theta_{g-b} \otimes \overline{F}_q,\Qlbar)\cong \HC^j(J \otimes \overline{F}_q,\Qlbar) \,\,\,\,\,\,\,\,\,\,\,\,\,\, \mbox{for}\,\,\,\,\, 0\leq j<g - a - b$$
The isomorphism preserves the generalized eigenspaces of Frobenius.
\end{corollary}
\begin{proof}
The isomorphism comes from Lemma \ref{thetathetaislci} and the Lefschetz hyperplane theorem. The equality of Frobenius eigenvalues follows as in Corollary \ref{stability1}.
\end{proof}

For a line bundle $L$, recall we write

\[e(L)  := \max \{m \, | \, \HC^0(C, \kappa^{-m} L) \ne 0 \}\]

\begin{lemma}\label{Esemismall}
Fix $L$.  Outside a set of dimension  $< e(L)$ on $\Sigma^{-1}(\Ll) = \Theta_{g-a} \cap \Ll - \Theta_{g-b}$, the map
$(A \times A)_{\Ll}: (A \times A)^{-1} \Sigma^{-1} (\Ll) \to \Sigma^{-1}(\Ll)$ is semismall and the relevant loci are
$\Theta_{g-a-2l} \cap L -  \Theta_{g-b-2r}$.
\end{lemma} \begin{proof} Define $\Theta^0_i = \Theta_i\backslash \Theta_{\le i}$. This gives a stratification of $\Theta$.
The fibre of $(A \times A)_{\Ll}$ over any point of $\Theta^0_{g-a-2l}\times
\Theta^0_{g-b-2r} \cap \Sigma^{-1}(L)$
is $\p^r\times\p^l$.

By Corollary \ref{cor:dimestimate},  as long as
$$e(L - (r+s) \kappa) = e(L) - l - r \le g-a-2l - b - 2r \iff e(L) \le g - a - b - l - r$$
the space $\Theta_{g-a-2l}\times \Theta_{g-b-2r} \cap s^{-1}(L)$ will be of the expected dimension $g-a-b-2(r+l)$.
In particular, in this case the condition for semismallness is not violated on $\Theta^0_{g-a-2l}\times
\Theta^0_{g-b-2r}$.

On the other hand, if the inequality fails to hold, then
\begin{eqnarray*}
e(L) & > & g - a - b - l - r \\ & \ge & \frac{(g-2a-2l) + (g-2b-2r)}{2} \\ & \ge & \min(\dim \Theta_{g-2a-2l}, \dim \Theta_{g-2b-2r}) \\
& \ge & \dim \Theta_{g-2a-2l} \cap L -  \Theta_{g-2b-2r}
\end{eqnarray*}
\end{proof}

\begin{proposition} \label{lem:bm2}
Assume $e(L) \le g-a-b$.  In the abelian category of perverse sheaves on $\Theta_{g-a} \cap L - \Theta_{g-b}$,
there is a morphism $\Qlbar[g-a-b] \to \mathrm{IC}_{\Theta_{g-a} \cap L - \Theta_{g-b}}$
and its kernel is supported
in dimension $\le e(L)$.
\end{proposition}
\begin{proof}
The assumption guarantees that $\Theta_{g-a} \cap L - \Theta_{g-b}$ has the expected dimension $g-a-b$.
Lemma \ref{thetathetaislci} asserts this locus is set theoretically l.c.i. variety, hence the constant sheaf is perverse, from which the existence of the map follows.

Lemma \ref{Esemismall} asserts that outside of dimension $< e(L)$ the map $A \times A: (A \times A)^{-1} s^{-1} (L) \to s^{-1}(L)$
is semismall; Corollary \ref{cor:dimestimate} asserts that outside of dimension $\le e(L)$ the source is nonsingular.  In other words,
restricting to the complement of a locus of dimension $\le e(L)$, we have a semismall
resolution of singularities for which the ``relevant'' loci are precisely the $\Theta_{g-a-2l} \times \Theta_{g-b-2r} \cap \phi^{-1}(L)$.
Over the open locus of each of these, the fibre of
is $\p^l\times\p^r$; note moreover that even if $r = l$ the factors are distinguished from each other.
Thus the local system of cohomologies is trivial of dimension $(l+1)(r+1)$.
Observe moreover that $\Theta_{g-a-2l} \times \Theta_{g-b-2r} \cap s^{-1}(L)$ lies in the closure of $\Theta_{g-a-2l'} \times \Theta_{g-b-2r'}
\cap s^{-1}(L)$ for all
$l' \le l$ and $r' \le r$.  As there are $(l+1)(r+1)$ such loci, the Borho-Macpherson trick
(as in the proof of Lemma \ref{lem:bm1}) allows us to conclude that the contribution of each such $\Theta_{g-a-2l'} \times \Theta_{g-b-2r'} \cap s^{-1}(L)$ is the
IC sheaf, which is moreover equal to the constant sheaf. \end{proof}

\begin{theorem} \label{cohomology} Fix $L \in J^{2g-a-b}$.  Assuming $e(L) \le g - a - b$,
$$\HC^i(\Theta_{g-a} \cap L-\Theta_{g-b}, \Qlbar) \cong \HC^{i+2a + 2b}(J,\Qlbar)(a+b) \,\,\,\,\,\,\,\,\, \mathrm{for}\,\,\,\,  i > g - a - b + e(L)$$
\end{theorem}
\begin{proof} By Lemma \ref{lem:bm2}, we have an exact sequence of perverse sheaves on
$\Theta_{g-a} \cap L - \Theta_{g-b} $:
\[0 \to F \to \Qlbar[g-a-b] \to \mathrm{IC}_{
\Theta_{g-a} \cap L - \Theta_{g-b}} \to 0\]
where
$F$ is supported in dimension $\le e(L)$.  As $F$ is perverse, we have by definition
that $\Hh^{-i}(F)$ is a sheaf supported on a variety of dimension $i$, and in particular that
$\HC^k(\Hh^{-i}(F)) = 0$ for $k \notin [0, 2i]$.  Thus from the hypercohomology spectral sequence
$$\HC^j(\Theta_{g-a} \cap L - \Theta_{g-b}, \Hh^{-i}(F)) \Rightarrow
\HC^{j-i}(\Theta_{g-a} \cap L - \Theta_{g-b}, F)$$ we see that
$\mathbb{H}^*(\Theta_{g-a} \cap L - \Theta_{g-b}, F)$
is supported in degrees between $\pm e(L)$.  From the long exact
sequence of hypercohomology we deduce
$$\HC^{i-(g-a-b)}(\Theta_{g-a} \cap L - \Theta_{g-b},\Qlbar)
\cong \mathrm{IH}^i(\Theta_{g-a} \cap L- \Theta_{g-b})\,\,\,\,\,\,\,\,\, \mathrm{for}\,\,\,\, i \notin [-e-1, e]$$
Corollary \ref{thetathetalef} identifies $\HC^{< g-a-b}(\Theta_{g-a} \cap L - \Theta_{g-b}, \Qlbar)$
with the corresponding cohomology groups of $J$.
Poincar\'e duality of the intersection cohomology gives the desired result.
\end{proof}

\subsection{Towards the exponential bound}

Consider the map $\Sigma: \Theta_{g-a} \times \Theta_{g-b} \to J^{2g-a-b}$.  In this section we restrict ourselves
to the locus
$$\widetilde{J}^{2g-a-b} := \{L \in J^{2g-a-b}\, | \, e(L) \le g - a - b \}$$

\begin{conj} \label{conj:expboundinter}
There exists a universal constant $c$ (independent of $g$, the field, the curve, $a$, $b$)  such that for all $y \in \tilde{J}$, we have
$\dim H^*(s^{-1}(y)) \le c^g$.
\end{conj}

By Lemma \ref{lem:bm1}, the total space $\Theta_{g-a} \times \Theta_{g-b}$ is a homology manifold, i.e. the constant
sheaf is the IC sheaf.  Therefore the decomposition theorem applies to the
map $\Sigma$, and we have

$$R\Sigma_* \Qlbar[2g-a-b] = \bigoplus {}^p R^i \Sigma_* \Qlbar[2g-a-b]$$

By Lemma \ref{thetathetalef},
for $i < 0$ the perverse sheaf ${}^p R^{i}s_*\Qlbar[2g-a-b]$ is just a constant sheaf
with fibre $\HC^{g+i}(J^{2g-a-b})$.  By the Relative Hard Lefschetz theorem, we have isomorphisms
${}^p R^{-i} \cong {}^p R^i (i)$.  As $\dim \HC^*(J) = 4^g$, it is enough to bound the stalks of the remaining term

$$F:= {}^p R^0 s_*\Qlbar[2g-a-b]$$

In the remainder of this section we prove

\begin{theorem} \label{overc}
Let $C \to \PP^1$ be a hyperelliptic curve of genus $g$ over $\C$.  Then for $L \in J^{2g-a-b}$ such that
$e(L) \le g - a - b$, we have  $\dim \HC^*(\Theta_{g-a} \cap L - \Theta_{g-b}, \Q) \ll_\epsilon (960 + \epsilon)^g$.
\end{theorem}

D. Massey has explained how to obtain Morse-type inequalities controlling the stalk cohomology
of a perverse sheaf in terms of data
 which only depends on the constructible function of Euler characteristics of the stalks,
 and furthermore how to extract this information
from polar varieties \cite{Ma}.   To state the result we recall that the local
Euler obstruction is a certain constructible function intrinsically defined on an
algebraic variety $V$, whose value at $p$ is the virtual number of zeroes of
any extension of the radial vector field on the link of $p$ to
its interior, after passing to the Nash blowup $\widetilde{V} \to V$
so this makes sense \cite{M}.  We write $\mathrm{Eu}_V$ for this
constructible function.  The Euler obstruction takes the value
$1$ at smooth points of $V$, so for a smooth variety $Y$ the set of Euler
obstructions of subvarieties gives an integral basis for the constructible
functions on $Y$.

We also recall the definition of polar varieties.
For $p \in V \subset Y$, we choose generic coordinates $y_0, y_1,\ldots$ near $p$.  Writing $V^{sm}$ for the smooth locus of $S$ and
recall that the polar varieties of $V$ (with respect to the fixed general coordinates near $p$)
are by definition
$$\Gamma^i_V:= \mbox{the closure of the union of the $i$-dimensional components of }
\mathrm{Crit}(V^{sm} \xrightarrow{y_0,\ldots,y_i} \C^{i+1})$$
The {\em polar multiplicities} $\gamma^i_V$ are  the constructible functions on $V$ such that
$\gamma^i_V(p) = \mathrm{mult}_p \Gamma^i_V$, where the polar variety is defined with respect
to any generic choice of coordinates.

\begin{theorem} \label{thm:massey} \cite{Ma}
Let $K$ be a perverse sheaf on a smooth space $Y$.  Expand the stalk Euler characteristics
in the basis of Euler obstructions, that is, determine the varieties $V_\alpha$ and coefficients $n_\alpha$
so that for all $p \in Y$ the following holds:
\[\chi(K_p) = \sum_\alpha n_\alpha \mathrm{Eu}_{V_\alpha}(p)  \]
Then for all $p \in Y$,
\[\dim \HC^{-i}(K_p,\Q) \le \sum_{\alpha} n_\alpha \gamma^i_{V_\alpha}(p) \]
\end{theorem}

In fact in \cite{Ma}, Massey's main interest is in constructing the ``characteristic polar cycles'', whose multiplicity at a point $p$ records precisely
$\sum_{\alpha} n_\alpha \gamma^i_{V_\alpha}(p)$.  For a proof
of the above result avoiding any mention of characteristic polar cycles, see \cite[Sec. 5]{MS2}.

Thus we must  determine, for $K = {}^p R^0 s_*\Q[2g-a-b]$,
the varieties $V_\alpha$, the multiplicities $n_\alpha$ with which they appear, and
then finally the polar multiplicities $\gamma^i_{V_\alpha}$.  A priori one might expect
the closure of any stratum in a (perhaps Whitney) stratification of $K$ to appear.  In fact, the situation
here is much better, due to the results of \cite{MS2}.

\begin{definition} \cite{MS2} Let $f: X \to Y$ be a proper map between smooth varieties.  We define the higher discriminants \[\Delta^i(f):= \{ y \in Y \, |\, \mbox{no
$(i-1)$-dimensional subspace of $T_y Y$ is transverse to $f$}\,\}\] \end{definition}

Note by generic smoothness that $\Delta^{i}(f)$ is of codimension at least $i$.

\begin{thm}
\label{thm:index}
 \cite{MS2} Let $f: X \to Y$ be a proper map between smooth varieties.  Let $\Delta^{i,\alpha}$ be the codimension $i$ irreducible components of $\Delta^i(f)$.  For $i \ge 1$,
let $x \in \Delta^{i, \alpha}$ be a general point, and $\D^i \ni x$ disc transverse to $\Delta^{i,
\alpha}$ such that $X|_{\D^i}$ is nonsingular.  Let $l: \D^i \to \A^1$ be a general linear function in order to define
$n_{i, \alpha} := \Phi_l (f_* 1_X|_{\D^i})$.  Let $X_{gen}$ be a general fibre.  Then
\[f_* 1_X = \chi(X_{gen}) \cdot 1_Y + \sum_{i \ge 1, \alpha} n_{i,\alpha}
\mathrm{Eu}_{\Delta^{i,\alpha}}\] \end{thm}

To apply this in the present situation we first replace the function $p \mapsto \chi(({}^p R^0 s_* \Q)_p)$ with
the function $\Sigma_* 1$; these differ by a constant function where the constant is bounded by $4^g$.
We temporarily restore the $a, b$ to the notation and write
$\Sigma^{a,b}: \Theta_{g-a} \times \Theta_{g-b} \to \widetilde{J}_{2g-a-b}$
and similarly
$\underline{A}^{a,b}: C^{(g-a)} \times C^{(g-b)} \to  \widetilde{J}^{2g-a-b}$.  We have an equality of relative motives
$$[\underline{A}^{a,b}] = \sum_{l,r \ge 0} \mathbb{L}^{r+s} [\Sigma^{a-2l, b-2r}] $$
and for that matter also an equality of constructible sheaves
$$\underline{A}^{a,b}_* \Q = \sum_{l,r \ge 0} \Sigma^{a-2l, b-2r}_* \Q[-2r-2s](-r-s)$$
either of which descend to an equality of constructible functions
$$\underline{A}^{a,b}_* 1 = \sum_{l,r \ge 0} \Sigma^{a-2l, b-2r}_*1$$

The motivic equality can be inverted:
\[ [\Sigma^{a,b}] =  [\underline{A}^{a,b}] - \mathbb{L} [\underline{A}^{a+2,b}] -  [\underline{A}^{a,b+2}] + \mathbb{L}^2
[\underline{A}^{a+2,b+2}] \]
and hence
\begin{equation} \label{eq:stoa}
 \Sigma^{a,b}_* 1 = \underline{A}^{a,b}_* 1- \underline{A}^{a+2,b}_* 1 - \underline{A}^{a,b+2}_* 1 + \underline{A}^{a+2,b+2}_* 1
 \end{equation}

In the subsequent subsections we calculate the quantities appearing in Theorem \ref{thm:index} and
Theorem \ref{thm:massey}.

\begin{proof}(of Theorem \ref{overc}) According to Proposition \ref{prop:euler}, the
absolute value of the Euler characteristic of the general fibre of $\underline{A}$ is bounded by $8^g$.
In Corollary \ref{cor:higher} we show that
the irreducible components of $\Delta^i(\underline{A})$ are among the varieties $\Theta_{t} + 2_* \Theta_u$, where
$t+u = g - i$.  In Proposition \ref{prop:nodecalc} we show
the corresponding coefficient of the expansion in Theorem \ref{thm:index} is bounded by $10^g$.
Finally in Proposition \ref{prop:polar} we show that the corresponding
polar multiplicities are bounded by $g^2 96^g$.  Taken together, this yields the stated bound on the
cohomology groups.
\end{proof}

We may also conclude, independently of the arguments in Section \ref{subsec:semismall}, that

\begin{proposition}
Fix $L \in J^{2g-a-b}$.  Assuming $e(L) \le g - a - b$,
$$\HC^i(\Theta_{g-a} \cap L-\Theta_{g-b}, \Q) \cong \HC^{i+2a + 2b}(J,\Q)(a+b) \,\,\,\,\,\,\,\,\, \mathrm{for}\,\,\,\,  i > g - (1/2)(a + b) + e(L)$$
\end{proposition}
\begin{proof}
$\mathcal{H}^{-j}({}^p R^0 s_* \Q)$
is supported on the union of the
$\Theta_t + 2 _* \Theta_u$ for $t + u = j$.  We have $\Theta_t + 2_* \Theta_u \subset \Theta_{t+2u} \subset \Theta_{2j}$.
So there is no $\mathcal{H}^{-j}$ so long as $L \notin \Theta_{2j}$.

In $J^{2g-a-b}$, we have $L \in \Theta_{2g-a-b -2E} \iff e(L) \ge E$, so
$L \notin \Theta_{2j} \iff e(L) < (2g - 2j - a - b)/2$, that is,
$j < (2g - 2e(L) - a - b)/2$.

Thus, for $j$ such that $\mathcal{H}^j\neq 0$ we have
$$(g-j)+(g-a-b)\geq (2g-a-b) - ((2g - 2e(L) - a - b)/2) = g-(1/2)(a+b) + e(L).$$
\end{proof}

Note this is weaker than (the specialization to characteristic 0) of Theorem \ref{cohomology},
but would still suffice for our purposes.

\subsection{Euler numbers of the general fibres}
\label{subsec:euler}

\begin{proposition} \label{prop:euler}
The absolute value of the Euler characteristic of the general fibre of $\underline{A}$ is bounded by $8^g$.
\end{proposition}

\begin{proof}
Let $F$ be a fibre of the map $\underline{A}: C^{(g-a)} \times C^{(g-b)} \to J_{2g-a-b}$ which is smooth
and of the expected dimension $g-a-b \ge 0$.  By smoothness and the triviality of $TJ$,
$$\chi(F) = c(TF) \cap [F] = \frac{c(T(C^{(g-a)} \times C^{(g-b)}))}{c(TJ)}  \cap [F] = c(T(C^{(g-a)} \times C^{(g-b)}))
\cap [F]$$

We recall standard facts about $H^*(J)$ and $H^*(C^{(n)})$, see, e.g., \cite[Chap. VIII.2]{ACGH}.  Fix a symplectic basis $\delta_1,\ldots, \delta_{2g}$ of $\HC^1(C)$, i.e.,
$\delta_i \delta_{g+i} = - \delta_{g+i}\delta_i = [pt]$, and all other products are trivial. We identify $\HC^1(C)= \HC^1(C^{(n)}) = \HC^1(J)$, the first equality being induced by the
inclusion $C^{(i)} \to C^{(j)}$ given by adding any fixed divisor of degree $j-i$, and the second by the Abel-Jacobi map $A:C^{(n)} \to J$. 
Then $\HC^*(J)$ is the exterior algebra on the $\delta_i$, with the class of the point given by $\delta_1 \delta_{g+1} \delta_2 \delta_{g+2} \ldots \delta_g \delta_{2g}$, and the
class of the theta divisor is given by \[\theta = \sum_{i=1}^g \delta_i \delta_{g+i}\]  We also write $\theta:=\pi^*\theta \in \HC^*(C^{(n)})$, and $x \in \HC^*(C^{(n)})$ for the class of
$C^{(n-1)}$.  The Chern class of the tangent bundle is given by \cite[p. 339]{ACGH}: \[c(TC^{(g-d)}) = (1+x)^{1-d} e^{-\theta/(1+x)}\]

On $C^{(g-a)} \times C^{(g-b)}$, the class $[F]$ is the Poincar\'e dual of
$\underline{A}^*([\mathrm{pt}]^{\vee})$. This we compute by factoring $\underline{A}$ as
$C^{(g-1)} \times C^{(g-1)} \xrightarrow{A \times A} J \times J \xrightarrow{s} J$, where the first map is Abel-Jacobi and the second is addition.  The pullback of the point
class under addition is the anti-diagonal, which by Kunneth is just \[(1 \otimes (-1))^* \sum_{I \subset \{1,\ldots,g\}} \delta_I \otimes \delta_{I^c} + \ldots\] Here $\delta_I := \prod_{i
\in I} \delta_i \delta_{g+i}$ and $I^c$ is the complement of $I$.  The terms in ``$\ldots$'' are those in which $\delta_i$ appears on one side of the tensor product and $\delta_{g+i}$
appears on the other; we are ultimately going to integrate against a power of $\theta$ and all such terms will integrate to zero.  Similarly $(-1)^*$ acts as $(-1)^i$ on $\HC^i$, but the odd
terms necessarily are in the ``$\ldots$''.  Thus we arrive at the formula \[\chi(F) = \sum_{I \subset \{1,\ldots,g\}} \left( \int_{C^{(g-a)}} \delta_I (1+x)^{1-a} e^{-\theta/(1+x)} \right) \left(
\int_{C^{(g-b)}} \delta_{I^c} (1+x)^{1-b} e^{-\theta/(1+x)} \right)\]
The classes are all pulled back from $J \times J$, with the exception of $x$.  We push to $J \times J$ using
Poincar\'e's formula: \[\pi_! [C^{(g-d)}] = \frac{\theta^{d}}{d!} \in
\HC^*(J)\,\,\,\,\,\,\,\,\,\,\,\,\,\,\,\,\,\, \mbox{for $0 \le d \le g$}\]
To evaluate integrals note
\[\int_J \theta^{g-|I|} \delta_I = (g-|I|)! \]

Let us calculate one of the factors.
$$  \int_{C^{(g-a)}} \delta_I (1+x)^{(1-a)} e^{-\theta/(1+x)}
= \sum_i \int_{C^{(g-a)}} \delta_I \frac{(-\theta)^i}{i!} (1+x)^{1-a-i} $$
\begin{eqnarray*}
& = & \sum_{i,n} \int_{C^{(g-a)}} \delta_I \frac{(-\theta)^i}{i!} {n+a+i-2 \choose n} (-x)^n \\
& = & (-1)^{g-a-|I|} \sum_n  {g-|I|-2 \choose n} \int_{C^{(g-a-n)}} \delta_I \frac{\theta^{g-a-n-|I|}}{(g-a-n-|I|)!}  \\
& = & (-1)^{g-a-|I|} \sum_n  {g-|I|-2 \choose n} \int_{J} \delta_I \frac{\theta^{a+n}}{(a+n)!}\frac{\theta^{g-a-n-|I|}}{(g-a-n-|I|)!}  \\
& = & (-1)^{g-1-|I|} \sum_{n} {g-|I|-2 \choose n} {g-|I| \choose n+a}
\end{eqnarray*}
This quantity has absolute value bounded by $4^{g-|I|}$, so
we conclude
$$ |\chi(F)| \le 8^g$$
This completes the proof.
\end{proof}

\subsection{Higher discriminants}
\label{subsec:higher}

We now determine the higher discriminants of the map $\underline{A}: C^{(g-a)} \times C^{(g-b)} \to J$.

\begin{lemma} \label{lem:higherdisc}
Consider $(D_1, D_2)$ a singular point in  $\underline{A}^{-1}(L) \subset C^{(g-a)} \times C^{(g-b)}$
with canonical decomposition $(R_1, R_2, R_\cap, \ldots)$; in particular
$L = \mathcal{O}(R_1 + R_2 + 2 R_\cap) \otimes \kappa^n$.   Let
$\mathcal{R}$ be the (closed) locus in $A^{-1}(L) \subset C^{(g-a)} \times C^{(g-b)}$
consisting of $(D'_1, D'_2)$ with canonical decomposition $(R'_1,R'_2, R_\cap, \ldots)$ such that $R'_1+R'_2 = R_1+R_2$.
Then the map $A$ is transverse along $\mathcal{R}$ to a generic $(g-\deg (R_1 + R_2 + R_\cap))$-dimensional subspace of $T_L J$, but to no smaller subspace. \end{lemma}

\begin{proof}
We  abbreviate $r:= \deg(R_1 + R_2 + R_\cap)$.
Recall that at $(D_1, D_2)$, we have
$$d\underline{A}\left( T_{(D_1, D_2)} (C^{(\cdot)} \times C^{(\cdot)}) \right)= dA \left( T_{D_1
\cup D_2} C^{(\deg D_1 \cup D_2)} \right)$$
Along $\mathcal{R}$, therefore,
$\mathrm{im} \,d\underline{A}$ always contains the space $\rho := dA(T_{R_1 + R_2 + R_\cap} C^{(r)})$,
so $\underline{A}$ is transverse along $\mathcal{R}$ to any complementary subspace;
these have dimension $g - r$.

Transversality of a proper map along a proper subvariety being an open condition, it is enough to show that
$\underline{A}$ is not transverse to any $V$ in a nonempty open subset of $G(g-1-r, T_0 J)$.
Consider the rational map $G(g-1-r, T_0 J) \dashrightarrow \p T_0 J^\vee$ given by $V \mapsto V \oplus \rho$.
Under the
identification $ \p T_0 J^\vee \cong (C / \tau)^{(g-1)}$ of Lemma \ref{embedding}, the image consists of
subschemes containing $R_1 + R_2 + R_\cap$; we consider the open set of the image of subschemes
of the form $p_1 + \ldots + p_{g-r-1} + R_1 + R_2 + R_\cap$ where the points $p_i$ are distinct
and also distinct from the points in $R_1, R_2, R_\cap$.
That is, the $(g-1)$-dimensional vector space $V \oplus dA(T_{R_1 + R_2 + R_\cap} C^{(r)})$
can be written as $\ell_{p_1} \oplus \ldots \oplus \ell_{p_{g-1-r}} \oplus  dA(T_{R_1 + R_2 + R_\cap} C^{(r)})$.

By assumption, $L$ has singular preimage, and so by Corollary \ref{cor:tdim} we have
the inequality
$$r + \frac{\deg H_1 + H_2 + H_\cap}{2} + \deg S \le g - 1$$
and in particular
$$ \deg H_1 + \deg H_\cap + \deg S \le 2(g-1-r)$$
The following divisor therefore exists:
$$(D^\circ_1, D^\circ_2) = (R_1 + R_\cap + p_1 + \overline{p_1} + p_2 + \overline{p_2} + \ldots, R_2 + R_{\cap} +
\overline{p_1} + p_1 + \overline{p_2} + p_2 +  \ldots) \in C^{(g-a)} \times C^{(g-b)}$$

By construction, $d\underline{A}\left( T_{(D^\circ_1, D^\circ_2)} (C^{(g-a)}\times C^{(g-b)}) \right) \subset V +
dA(T_{R_1 + R_2 + R_\cap} C^{(r)})$, so $\underline{A}$ is not transverse to $V$ at
 $(D^\circ_1, D^\circ_2)$.
\end{proof}

\begin{notation} We write $\mathbf{2}: J \to J$ for the multiplication by 2 map.  For $X \subset J$, we write
${\bf 2} X := {\bf 2}(X)$ and ${\bf \frac{1}{2}} X := {\bf 2}^{-1} (X)$. \end{notation}

\begin{corollary}\label{cor:higher}
The irreducible components of $\Delta^i(A)$ are among the $\Theta_r + {\bf 2}\Theta_s$ satisfying $r + s = g -i$.
\end{corollary}
\begin{proof}
Take $r = \deg R_1 + \deg R_2$ and $s = \deg R_\cap$.
\end{proof}

\begin{proposition}\label{prop:nodecalc}
Let $\ell \in \Theta_r+{\bf 2}\Theta_s$ be a generic point.  Let $p + D_3$ be a generic divisor of degree $g-r-s$ consisting of distinct points with no hyperelliptic pairs; let $\ell \in \D^{g-r-s-1} \subset J$ be a disc with tangent space
$dA \left( T_{D_3} C^{(g-r-s-1)} \right)$.  Then the singularities of the fiber $\underline{A}^{-1}(U)$ are ordinary
double points, and there are at most $10^g$ of them.
 \end{proposition}

 \begin{proof}
Since $\ell$ is generic, it is represented as $D_1+2D_2$ in a unique way with $D_1\in C^{(r)}, D_2\in C^{(s)}$; choosing $D_3$ generically, we can assume that
$D_1 + D_2 + D_3$ is a set of distinct points.

Now, note that the tangent space image at $R_1+R_{\cap}+S+H_1+H_{\cap}$ of $A$ is contained in the image at $D_1+D_2+D_3$. Thus, by lemma \ref{tangents} we have
$$H^0(C,K-(D_1+D_2+D_3))\subset H^0(K-(R_1+R_{\cap}+S+H_1+H_{\cap})).$$ Since $\deg(D_1+D_2+D_3)=g-1<g$, it follows from lemma \ref{lem:kappasections} that
$$R_1+R_{\cap}+S+H_1+H_{\cap}\subset D_1+D_2+D_3+\overline{(D_1+D_2+D_3)}$$ and thus consists of distinct points. Likewise for $R_2+R_{\cap}+\bar{S}+H_2+H_{\cap}$

 As in the proof in Lemma \ref{lem:higherdisc}, singularities occur when $S+H_1+H_2+H_{\cap}$ is a subdivisor of $D_3+\overline{D_3}$.  The number of ways this can happen: first $D_1$ must
be separated into $R_1$ and $R_2$; this can be done in at most $2^g$ ways.  Then, $S, H_1, H_2, H_\cap$ must
be chosen from $D_3+\overline{D_3}$, giving $5^g$ choices.

To see that the singularities are ordinary double points, let $(V, W)$ be a singular point over $\ell$ in $\underline{A}^{-1}( \D^{g-r-s-1})$, with canonical decomposition $(R_1,R_2,R_{\cap},\dots)$.  The singularity is analytically isomorphic to
the singularity at $(V, W,D_3)$ of the pre-image at $\ell-D_3$ of  $$A':C^{(g-a)}\times C^{(g-b)}\times C^{(g-r-s-1)}\rightarrow J_{3g-r-s-a-b-1}.$$
By construction, $\dim \,\mathrm{im} \, dA' \left(T_{(V, W, D_3)} C^{(g-a)}\times C^{(g-b)}\times C^{(g-r-1)} \right) = g-1$.
We have seen above that each of $V, W, D_3$ are sums of {\em distinct} points, so analytically locally we may desymmetrize: the singularity is the same as the singularity at any point $P$ mapping to $W + V + D_3$ in the fibre of the map
$$B': C^{g-a} \times C^{g-b} \times C^{g-r-1} \to J$$

We select out a $C^{r}$ from the first two factors to account for $R = R_1 + R_2 + R_\cap$, and put it in the third factor:

$$B': C^{2g-a-b-r} \times C^{g-1} \to J$$

%
%
%
%

We write $P = (S, R) \in C^{2g-a-b-r} \times C^{g-1}$ for the point of interest.  Note the divisor of $R$
is $R_1 + R_2 + R_\cap$, which is a sum of distinct nonhyperelliptic points.  Thus $dA(T_R C^{g-1})$ is $g-1$ dimensional.
Let $\eta_1 \in \HC^0(J, \Omega_J) = \HC^0(C, K)$ be the unique differential form which vanishes along
this hyperplane;  $\eta_1$ spans $\HC^0(C,K-R) = \HC^0(C,K-R_1-R_2-R_{\cap}-D_3)$.  Let $\eta_2, \ldots, \eta_g$
form the remainder of a basis for $\HC^0(J, \Omega_J)$.  Locally near $B'(S, R) \in J$ we integrate these differential forms to give a map
to $\C^g$.  This being a complex analytic isomorphism, the analytic type of the singularity in the preimage remains unchanged.

Note that mapping $C^k$ to the Jacobian and then integrating an element $\xi \in \HC^0(J, \Omega_J)$ near the image of $(p_1,\ldots,p_k)$
is just the same as $(q_1,\ldots, q_k) \mapsto \sum_i \int_{p_i}^{q_i} \xi$, where $\xi$ is the corresponding element in $\HC^0(C, K)$.
Thus we may study the preimage of $0$ in the map

$$\int \eta_1 , \ldots,   \int \eta_g :  C^{2g-a-b-r} \times C^{g-1}  \to \C \times \C^{g-1}$$

By construction, the restricted map $C^{(g-1)} \to \C^{g-1}$ has nondegenerate Jacobian at $R$.  Therefore, by the implicit function
theorem, the singularity of the above map at $S + R$ is the same as the singularity at $S$ of

$$\int \eta_1: C^{2g-a-b-r} \to \C$$.

For this map to have a singularity at $S$, we must have $d \int \eta_1 = \eta_1$ vanish at every coordinate of $S$.
But by Lemma
\ref{lem:kappasections}, $\eta_1(p) = 0 \iff p \in R_1+R_2+R_3+D_3$. But in this case,
$$\HC^0(C,K-R_1-R_2-R_{\cap}-D_3 -p)=H^0(C,\overline{R_1}+\overline{R_2}+\overline{R_{\cap}}+\overline{D_3}-p)=0$$
by Lemma \ref{hyperpair} and the generic choice of $D_3$; it follows that $\omega$ has only a simple zero at $p$.
Thus we see that the Hessian of the map $\int \eta_1$ is nondegenerate, and consequently that its singularity is an ordinary double
point.
\end{proof}

\subsection{Polar multiplicities} \label{subsec:polar}

We recall the construction of polar varieties, see e.g. \cite{Rag, LT, Kl2}.
For any $X \subset J$,  the conormal scheme $\overline{N^*_X J} \subset
\p T^* J = J \times \p T_0 J^\vee$ is by definition the closure of the locus
$(x, \xi)$ where $x \in X^{sm}$ and $\xi$ viewed as a cotangent vector annihilates $T_x X$, or
viewed as a point in $\p T_0 J^\vee$ and hence a hyperplane in $\p T_0 J$, contains $T_x X$.

View $\overline{N^*_X J}$ now as a correspondence
$$ J \xleftarrow{\pi} \overline{N^*_X J} \xrightarrow{\pi^\vee} \p T_0 J^\vee$$
Then for any vector subspace $L \subset T_0 J$, we define the polar variety

$$P_L X := \pi(\overline{N^*_X J} \cap (\pi^{\vee})^{ -1}(L^\vee) ) $$

As in \cite{LT}, by the Kleiman-Bertini theorem \cite{Kl} we have that for generic $L$, the space $P_L X$ is the same
as the closure in $X$ of the locus of $x \in X^{sm}$ where $T_x X$ is not transverse to $L$.
That is, for generic $L$ of dimension $k$, the space $P_L X$ is what we previously called $\Gamma^{g-1-k}_X$;
we have restored the $L$ to the notation as we intend to vary it.

In our setting, the polar varieties may be obtained by a different construction. Consider the correspondence

 $$ J\xleftarrow{\pi_{r,s}}C^{(r)}\times C^{(s)}\times (C/\tau)^{(g-1-r-s)}\xrightarrow{\sigma} (C/\tau)^{(g-1)} \cong \p T_0 J^{\vee}.$$

We write $P'_L V_{r,s}: = \pi_{r,s}\left(\sigma^{-1}(L^\vee)\right)$.

\begin{lemma}\label{polarsame}

Let $B \subset (C/\tau)^{(g-1)}$ be the union of the discriminant locus and the $2g+2$ hyperplanes of divisors which
contain Weierstrass points. For  $L\subset B$, we have $P'_L V_{r,s} \subset P_L V_{r,s}$, and moreover for generic $L$ we
have equality.
\end{lemma}

\begin{proof}

Consider the open subset $W_{r,s}\subset C^{(r)}\times C^{(s)}$  consisting of points  $\left(\sum_{i=1}^r P_i,\sum_{j=1}^s Q_j\right)$ such that the set $\{P_1,\dots,P_r,Q_1,.\dots,Q_s\}$ contains no Weierstrass points, no hyperelliptic pairs, no repeated points, and is contained in $\phi_{r,s}^{-1}(V_{r,s}^{sm})$. Set $U_{r,s}=\phi_{r,s}(W_{r,s})$. By definition, if $(D_r,E_s) \in W_{r,s}$ then $\phi_{r,s}(D_r,E_s)\subset P_L U_{r,s}$ iff $d\phi_{r,s}(D_r,E_s)$ is contained in some $g-1$ dimensional space $V \in L^{\vee}$. Under the identification $\p T_0J^{\vee}\cong |K_C|$,
$$\sigma: D \mapsto \HC^0(C,K_C-\pi^*D).$$

We write $\pi_{1,2}$ for the projection $C^{(r)}\times C^{(s)}\times (C/\tau)^{(g-1-r-s)} \to
C^{(r)}\times C^{(s)}$.
Then, $(D_r,E_s)\in (\pi_{1,2}\circ \sigma^{-1}(\omega))$ iff $\omega$ vanishes on $D_r+E_s+\bar{D_r}+\bar{E_s}$.
In particular, if $(D_r,E_s)\in W_{r,s}$ this is equivalent to  $\ell_p\in [\omega]$ for $p\in D_r\cup E_s$, where $[\omega]$ is the hyperplane corresponding to $\omega$.  Note that by construction, if $x \notin B$ then $\pi_{1,2}( \sigma^{-1}(x))\subset W_{r,s}$.

Thus, $$\pi_{r,s}( \sigma^{-1}(L^{\vee}))\cap U_{r,s} = P_L U_{r,s}=P_L V_{r,s}\cap U_{r,s}.$$

Now, consider $L^{\vee}\not\subset B$.  Because $\sigma$ is flat, all the associated points of ${\pi_{r,s}^{-1}(L^{\vee})}$ map to the generic point of $L^{\vee}$ and thus $\overline{\pi_{r,s}^{-1}(L^{\vee} \setminus B)} = \pi_{r,s}^{-1}(L^{\vee})$.
As the maps $\sigma$ and $\pi_{1,2}$ are proper,
$$ P'_L V_{r,s} = \pi_{r,s} (\sigma^{-1}(L^{\vee})) = \pi_{r,s} (\overline{\sigma^{-1}(L^{\vee} \setminus B)})
= \overline{\pi_{r,s}(\sigma^{-1}(L^{\vee} \setminus B))} = \overline{P_LU_{r,s}}.$$

As $\overline{P_LU_{r,s}}\subset P_L V_{r,s}$ this establishes the first claim. By Kleiman-Bertini, as in \cite{LT}, we generically have
$\overline{P_LU_{r,s}}= P_L V_{r,s}$, giving the second claim.
\end{proof}

\begin{lemma}\label{polarclass}
For generic $L$ of dimension $k$,
$$[P'_L V_{r,s}] = \pi_{r,s}(\sigma^{-1}[L^\vee]) = c_{k,r,s}\cdot[\Theta_{g-1-k}] \in \HC^*(J)$$ where
$$c_{k,r,s}= \!\!\!\!\! \sum_{\substack{a+b=k-(g-1-r-s)\\ a\le r, b \le s}} \!\!\!\!\!\! 2^{a+b} \binom{k}{a,b,g-1-r-s} 2^{s-b} \binom{r+s-a-b}{r-a}\leq
g^2 24^g$$
\end{lemma}

\begin{proof}

We consider first the map $\overline{\sigma}:\PP^{(r)}\times \PP^{(s)}\times\PP^{(g-1-r-s)}\rightarrow\PP^{(g-1)}$
given by adding divisors on $\PP^1$.  It is clear that
$$\overline{\sigma}^{-1}(H)=H \boxtimes 1 \boxtimes 1 + 1 \boxtimes H \boxtimes 1 + 1 \boxtimes 1 \boxtimes H.$$

Under the map $\pi^{(r)}: C^{(r)}\rightarrow (C/\tau)^{(r)}\cong \PP^r$, we have $(\pi^{(r)})^{-1}(H^k) = 2^k [C^{(r-k)}]$, so
$$\sigma^{-1}[L^\vee] = \sigma^{-1}[H^k] = \sum_{a+b+c=k} 2^{a+b}\binom{k}{a,b,c} [C^{(r-a)}]\boxtimes [C^{(s-b)}]\boxtimes H^c.$$

Denote the addition map $\Sigma:J\times J\rightarrow J$.
$$( \pi_{r,s})_*([C^{(t)}] \boxtimes [C^{(u)}]) =  \Sigma_*( A_* [C^{(t)}] \boxtimes {\bf 2} A_*[C^{(u)}] ) = 2^u \Sigma_* ([\Theta_t] \boxtimes [\Theta_u])
= 2^u\binom{t+u}{u} [\Theta_{t+u}]$$.

Thus
$$\pi_{r,s}(\sigma^{-1}[L^\vee]) = [\Theta_{g-1-k}] \!\!\!\!\! \sum_{\substack{a+b=k-(g-1-r-s)\\ a\le r, b \le s}} \!\!\!\!\!\! 2^{a+b} \binom{k}{a,b,g-1-r-s} 2^{s-b} \binom{r+s-a-b}{r-a}.$$
\end{proof}
\vspace{10mm}


\begin{proposition}\label{prop:polar}
Fix  $E \in V_{r,s}$.  For a general $L$ of dimension $k$,
the multiplicity at $E$ the polar variety of $P_L V_{r,s}$ is at most $g^2 96^g$.
\end{proposition}

\begin{proof}
Let $Q$ be a general point of $\Theta_{g-r-s}$, let $L$ be a line bundle with $2L=Q$.
We probe the multiplicity at $E$ by intersecting $P'_L V_{r,s}$
with $E-L+{\bf \frac{1}{2}}(\Theta_{k+1})$.  If $E$ is an isolated point of the intersection,
we may estimate the multiplicity by the contribution to the intersection multiplicity \cite[Thm. 12.4]{Ful}.
Since we are in an abelian variety, every connected component of
the intersection contributes non-negatively\footnote{As explained to us both by
R. Lazarsfeld and by ``ulrich'' on mathoverflow:  in the intersection theory of Fulton and Macpherson, the contribution
of any component $W$ of $X \cap Y$ inside $Z$ is given by intersecting the cone supported on $W$ with the zero section
of $TZ$.  Thus if $TZ$ is nef (in the present case the tangent bundle is trivial) this contribution is non-negative.  See
\cite[Thm. 12.2 (a)]{Ful}.} and thus we may bound
the multiplicity by the total intersection number.  So if $E$ is an isolated point of the intersection,
then by Lemma \ref{polarclass} we have
$$
 \mathrm{mult}_E(V_{r,s}) \le [E-L+\mbox{${\bf \frac{1}{2}}$}\Theta_{g-r-s}]\cap [P'_L V_{r,s}]
 \le 2^g[\Theta_{k+1}]\cap c_{\dim L, r,s}[\Theta_{g-1-k}] \le  g^2 96^g .$$

To conclude that $E$ is an isolated component of
$P'_L V_{r,s} \cap \left( E-L+{\bf \frac{1}{2}}\Theta_{k+1}\right) $, it is enough, by upper
semicontinuity of dimension in algebraic families,
to check that it is an isolated component of the intersection for some specific $L$.
We choose this $L$ as follows.
Let $\{p_1,\dots,p_k\}$ be distinct, non-Weierstrass points, containing no hyperelliptic pairs. Set $D=\sum_i p_i$ and
$L = dA(T_D C^{(k)})$. Then $\pi_{1,2}(\sigma ^{-1} L^{\vee})$ consists of the locus in $C^{(r)}\times C^{(s)}$ of divisors containing at least 1 point in at least
$k+r+s+1-g$ of the pairs $\{p_i,\bar{p_i}\}$. Hence $P'_L V_{r,s}$ consists of the union of $2^{k+r+s+1-g}$ translates of subsets $V_{t, u}$ for $t + u=g-k-1$.

Thus, it suffices to show that $E$ is an isolated point of
$(E-L+{\bf \frac{1}{2}}\Theta_{g-t-u} ) \cap  V_{t, u}$.
Suppose instead the intersection contains a curve $X \ni E$, and
consider the map from its normalization $\phi: \widetilde{X} \to J$.
Viewing $C, {\bf \frac{1}{2}}C$ as subsets of $J$,
we have a proper surjective map $({\bf \frac{1}{2}} C)^{g-t-u} \to
{\bf \frac{1}{2}} \Theta_{g-t-u} \supset X$ and so may lift $\phi$ to
$(f_1,\ldots, f_{g-t-u}): \widetilde{X} \to ({\bf \frac{1}{2}} C)^{g-t-u}$.

Similarly we may lift $\phi$ to
$(g_1,\ldots,g_t;h_1,\ldots,h_u): \widetilde{X} \to C^{t} \times C^u$.
We have
\begin{equation} \label{ela} E - L + \sum f_i(x) = \sum g_j(x) +2 \sum h_k(x) \,\,\,\,\,\,\,\,\,\,\,\,\,\,\,\,\,\,\,\,\,
\mbox{for all}
\,\, x \in \widetilde{X} \end{equation}
Taking a derivative,
we find a linear relation among tangent vectors:
\begin{equation} \label{elb} \sum df_i(x) = \sum dg_j(x) +2 \sum dh_k(x) \,\,\,\,\,\,\,\,\,\,\,\,\,\,\,\,\,\,\,\,\,
\mbox{for all}
\,\, x \in \widetilde{X} \end{equation}

Being general, $Q$ has a unique expression of the form
$Q = \sum_{i=1}^{g-t-u} Q_i$ for distinct $Q_i$.  It follows that
for $x$ in a Zariski open subset $U \subset \widetilde{X}$, the maps $f_i(x)$ take distinct values.
 At least one of the $f_i$, say $f_1$ must be nonconstant; thus on a Zariski open subset
$V \subset X$, we have $df_1 \ne 0$.  Let $x \in U \cap V$.

Recall for a point $p \in C$, we write $\ell_p \in \p T_0 J$ for the line spanned
by the image of $T_p C$ under the Abel map.
In these terms,
we have $\mathrm{im}\, df_i(x) \subset \ell_{2 f_i (x)}$, $\mathrm{im}\, dg_j(x) \subset \ell_{g_j(x)}$ and
$\mathrm{im}\, dh_k(x) \subset \ell_{h_k(x)}$.

We pick out the $g_i, h_k$ which are identically equal to
$2f_1$ or $\overline{2 f_1}$.  Reindexing as necessary, assume that $
\{g_j, h_k\} \cap \{2 f_1, \overline{2f_1} \}= \{ g_1,\ldots, g_\tau, h_1,\ldots,
h_\sigma\}$.  By Lemma \ref{tangents}, $df_1(x) \in \ell_{2f_1(x)}$
is not in the linear span of
$\{dg_{\rho + 1}(x), \ldots, dg_t(x), dh_{\sigma+1}(x),\ldots dh_u(x)\}$; it follows that

$$d f_1(x) = \sum_{j=1}^\sigma dg_j(x) + 2 \sum_{k=1}^\rho dh_k(x)$$

However, since these $g_j, h_k$ are each identified with $2f_1$ or $\overline{2f_1}$,
the right hand side is an even multiple of $df_1(x)$, which is a contradiction.
\end{proof}



\section{Equidistribution} \label{sec:eq}

We now collect the results of the article to prove the assertions in the introduction.
Because we are working with non simply connected groups, it is more natural
to split our various symmetric spaces into components enumerated by the corresponding fundamental groups.

Recall that, for the Jacobian of a hyperelliptic curve (a symmetric space for a non-split torus in $\PGL_2$), we have
used $\kappa = \pi^* \oO(1)$ to
identify $J^i \cong J^{i + 2}$.  In this section we preserve this identification, and generally take
$J^{g-1}$ and $J^{g-2}$ as representatives of the line bundles of holomorphic Euler characteristic of even and odd parity, respectively. Similarly for
$\mathrm{Bun}_2(\PP^1)$, we had identified (for us $0 \in \N$)
\begin{eqnarray*}
\mathrm{Bun}_2(\PP^1) & \leftrightarrow & \N \\
\Oo(a) \oplus \Oo(b) & \leftrightarrow &  |a-b|
\end{eqnarray*}
We now further separate this according to the parity of $|a-b|$, so that
$ \mathrm{Bun}_2^0(\PP^1) \leftrightarrow 2\N$ and
$ \mathrm{Bun}_2^1(\PP^1) \leftrightarrow 2\N + 1$.

\begin{lemma}
Let $\pi: C \to \PP^1$ be a hyperelliptic curve.
With the above identifications,
the pushforward map
$J(C) \to \mathrm{Bun}_2(\PP^1)$ is:
\begin{eqnarray*}
\pi_*:\mathrm{Pic}^{g-1}(C)& \rightarrow & 2\N \\ \Ll & \mapsto & 2 \dim \HC^0(C, \Ll) \end{eqnarray*}
and
\begin{eqnarray*}
\pi_*:\mathrm{Pic}^{g-2}(C)& \rightarrow & 2\N + 1  \\ \Ll & \mapsto & 2 \dim \HC^0(C, \Ll) + 1 \end{eqnarray*}
\end{lemma}
\begin{proof}
Since $\pi$ is finite, $\HC^i(C,\Ll) =
\HC^i(\p^1, \pi_*\Ll)$.
Say $\Ll \in \mathrm{Pic}^{g-1}(C)$.  Then by Riemann-Roch, $0 = \chi(\Ll) = \chi(\pi_* \Ll)$ so we may write
$\pi_*(\Ll) = \Oo(-n) \oplus \Oo(n-2)$ for $n \ge 1$.  Then $$2\HC^0(C,\Ll) = 2\HC^0(\p^1,\pi_*L) = 2\HC^0(\p^1, \Oo(-n) \oplus \Oo(n-2)) = 2(n-1) = |n-2 - (-n)|$$ In particular, $\Ll$ has no sections iff $n = m = 1$.

On the other hand, suppose
 $\Ll \in \mathrm{Pic}^{g-2}(C)$.   Riemann-Roch, $-1 = \chi(\Ll) = \chi(\pi_* \Ll)$, so we may write
 $\pi_*(\Ll) = \Oo(-n)  \oplus \Oo(n-3)$ for $n \ge 2$.  Then $\HC^0(C, \Ll) = n-2$, while
 $|n-3 - (-n)| = 2n - 3$.
\end{proof}

Up to normalization, the natural measure on $\mathrm{Bun}_2$ assigns to each point the inverse of the number of automorphisms of the
corresponding vector bundle.
The normalized natural measure is characterized by
$$\mu(d + 2 \N) =  \frac{1}{2q^{d-1}} \,\,\,\,\,\,\,\,\,\,\,\,\,\, \mathrm{for }\,\, d > 0 $$
Note in particular $\mu(\mathrm{Bun}_2^0(\PP^1)) = \mu( \mathrm{Bun}_2^1(\PP^1)) = 1/2$.

%
%

\begin{thm} \label{thm:theta}  Assume $q > 4$; we work over $\F_q$.
    Let $\pi_i:C_i \to \PP^1$ be a sequence of hyperelliptic curves, each carrying a line bundle
    $\Mm_i$.
     Let $\mu_i$ the pushforward of the Haar measure
    on $\mathrm{Pic}(C_i)/\pi_i^* \mathrm{Pic} (\PP^1)$ to $\mathrm{Bun}_2(\PP^1)$ under the map
    $\Ll \mapsto \pi_*(\Ll \otimes \Mm_i)$.   If no curve appears infinitely many times, then
     the measures $\mu_i$ converge to the natural measure on
    $\mathrm{Bun}_2(\PP^1)$.
\end{thm}
\begin{proof}
As there are only finitely many curves over $\F_q$ of any given genus, the limit amounts to a limit as $g \to \infty$.
The measure given by
$\Ll \mapsto \pi_*(\Ll \otimes \Mm_i)$ is the same as the measure given by
$\Ll \mapsto \pi_* \Ll$, so we assume wlog that the $\Mm_i$ are all trivial.
Thus we must show $\mu_i(d + 2 \N) \to q^{1-d}/2$ for $d > 0$.  For $d = 1$ this is just the statement that
$1/2 = \mu_i(1 + 2\N) = \mu(J^{g-2}(C))$, which holds because $C$ has a point over $\F_q$.   For $d > 1$,
we have $\mu_i(d + 2 \N) = \mu( \Theta_{g-d + 1} )$.
Thus it remains to show
\begin{equation}
\label{eq:pgl2limit}
\lim\limits_{g\to \infty} \frac{\# \Theta_{g-d+1} (\F_q) }{\# J_{g-d+1}(C)(\F_q)} = q^{1-d}
\end{equation}
We will give two proofs of this fact.

{\bf First proof}.
Let $C^{(n)}_0 \subset C^{(n)}$ be the locus of divisors with no hyperelliptic pairs.
In the Grothendieck ring
of varieties we have

\[Z_C(t) = \sum t^n C^{(n)} = \left(\sum t^{2n} \p^n \right)\left( \sum t^n C^{(n)}_0 \right) = \frac{1}{(1-t^2)(1-t^2 \mathbb{A}^1)} \left( \sum t^n C^{(n)}_0 \right) \]

Thus we can write
\begin{equation} \label{zetaresidue}
C^{(n)}_0= \mathrm{res}_{t=0}\left( t^{-n-1}(1-t^2)(1-t^2\mathbb{A}^1) Z_C(t)\right). \end{equation}
We now count points on both sides, and we abuse notation to still write $Z_C(t)$ for the point-counting zeta-function. Recall that $Z_C(t)=\frac{P_C(t)}{(1-t)(1-qt)}$  where $P(t)$ is a
polynomial of degree $2g$ that has the following properties: \begin{itemize} \item $P_C(t)$ has constant term 1 \item All roots of $P_C(t)$ have absolute value $q^{-\frac12}$, and thus
$P_C(t)=(t^2q)^gP_C(\frac{1}{qt})$. \item $P_C(1)=\#J(\F_q)=q^gP_C(q^{-1}).$ \item By \cite[Lem. 3]{AT}, on the circle $|t|=q^{-\frac12}$ we have the estimate $P_C(t)=O_{\epsilon}(q^{\epsilon g})$\end{itemize}

Going back to equation \eqref{zetaresidue}, we compute the residue by taking an integral around the circle $|t|=q^{-\frac12}$ and subtract off the residue at $t=\frac{1}{q}$. Thus
\begin{align*} \#C^{(n)}_0(\F_q)&=\frac{1}{2\pi i}\int_{|t|=q^{-\frac12}}t^{-n-1}(1+t)\frac{P(t)(1-qt^2)}{1-qt}-\mathrm{res}_{t=q^{-1}}
\left(t^{-n-1}(1+t)\frac{P(t)(1-qt^2)}{1-qt} \right) \\
&=O_{\epsilon}(q^{(\frac{n}2+\epsilon)\cdot g}) +q^n(1+q^{-1})(1-q^{-1})P(q^{-1})\\ &=\#J(\F_q)\cdot\left( q^{n-g}(1-q^{-2})+o(q^{-g/2+\epsilon\cdot g})\right)\\ \end{align*}
Finally we have $\# \Theta_n (\F_q) = \sum_{i=0}^{\lfloor n/2 \rfloor} C_0^{(n-2i)}$; summing the series yields the claim.

\vspace{2mm} {\bf Second proof.}
 Recall that for any variety
$X_{/\F_q}$, we have $$\# X(\F_q) = \sum (-1)^i \trf \HC_c^i (X \otimes \overline{\F}_q, \Qlbar).$$  Since $J, \Theta_j$ are compact, the eigenvalues of Frobenius on $\HC^i$ are bounded i absolute value by $q^{i/2}$ \cite{D}.
Thus to compare the point counts it suffices to compare the higher cohomology groups of $J$ and $\Theta_j$, and to bound the total dimension of the lower cohomology groups.   We have seen in Theorem \ref{thm:thetatop} that
$\HC^j(\Theta_i, \Qlbar)\cong \HC^{j+2g-2i}(J,\Qlbar)(g-i)$ for all $j > i$.  On the other hand we had
$\HC^j(\Theta_i, \Qlbar)\cong \HC^{j}(J,\Qlbar)$ for $j < i$.

The argument in the first proof can be used to compute the dimension of the middle cohomology
$\HC^i(\Theta_i, \Qlbar)$ explicitly.
However, we may produce the  bound $\dim \HC^*(\Theta_i)\ll 4^g$ by a slightly softer argument.
Indeed, since $C^{(i)} \to \Theta_i$ is semismall, it is immediate from the decomposition theorem
\cite{BBD} that $\HC^*(\Theta_i)$ is a summand of $\HC^*(C^{(i)})$.  The latter is known explicitly and satisfies the stated bound.

Collecting these estimates and comparisons, we have:
 \begin{align*} \frac{\#\Theta_{n}(\F_q)}{\#J(\F_q)}&=\frac{\sum_{i=0}^{2n}(-1)^i \trf \HC^i(\Theta_{n},\Qlbar)}{\sum_{i=0}^{2g}(-1)^i \trf \HC^i(J,\Qlbar)}\\
&=\frac{\sum_{i=n+1}^{2n}(-1)^i \trf \HC^i(\Theta_n,\Qlbar)+O(4^g\cdot q^{\frac{g}{2}})}{\sum_{i=n+1}^{2n}(-1)^i \trf \HC^i(\Theta_n,\Qlbar)(n-g)+O(4^g\cdot q^{\frac{g}{2}})}\\ &=q^{n-g}+O(4^g\cdot
q^{-\frac{g}{2}}). \end{align*}

The above  implies the result as soon as $q>16$.
\end{proof}

We turn to the case of $G = \PGL_2 \times \PGL_2$.  We are interested in the following sort of maps.

\begin{eqnarray*}   \mathrm{Pic}(C)/\pi^* \mathrm{Pic} (\PP^1)  & \to & \mathrm{Bun}_2(\PP^1) \times \mathrm{Bun}_2(\PP^1) \\
M & \mapsto & ( \pi_* (M \otimes L), \pi_* (M \otimes L') )
\end{eqnarray*}

To conclude a comparison on point counts, we require a bound on the lower cohomologies.  We believe

\begin{conj} \label{bound} There exist universal constants $N> 0$ such that for any curve of genus  $g \gg a,b$,
and any line bundle $L$ with $e(L) \le g - a - b$,
we have the bound $$\dim \HC^*(L - \Theta_{g-a} \cap \Theta_{g-b}) < N^g$$ \end{conj}

Note that in Theorem \ref{overc}, we established this statement for hyperelliptic curves over $\C$, with
$N = 960$.  In fact we believe the statement without the assumption on $e(L)$.

\begin{theorem} \label{thm:thetatheta}
Assume Conjecture \ref{bound} holds for some given $N$.  We work over $\F_q$ for some $q > N^4$.
Let $\pi_i:C_i \to \PP^1$ be a sequence of hyperelliptic curves, each carrying line bundles
$\Mm_i, \Mm_i'$.
Let $\mu_i$ the pushforward of the Haar measure
under the map
\begin{eqnarray*}
     \mathrm{Pic}(C_i)/\pi_i^* \mathrm{Pic} (\PP^1) & \to & \mathrm{Bun}_2(\PP^1) \times \mathrm{Bun}_2(\PP^1) \\
    \Ll & \mapsto & \left( \pi_*(\Ll \otimes \Mm_i), \pi_*(\Ll \otimes \Mm_i') \right)
\end{eqnarray*}
If no curve appears infinitely many times, and  for each $n \in\N$, there exists $A(n)$ such that for $i>A(n)$, $L_i \notin \Theta_{n}$,
then some the measures $\mu_i$ converges to the natural measure on one of
\begin{enumerate}
\item[(0)]  $\mathrm{Bun}_2^0(\PP^1) \times \mathrm{Bun}_2^0(\PP^1) \coprod \mathrm{Bun}_2^1(\PP^1) \times \mathrm{Bun}_2^1(\PP^1)$
\item[(1)]  $\mathrm{Bun}_2^0(\PP^1) \times \mathrm{Bun}_2^1(\PP^1) \coprod \mathrm{Bun}_2^1(\PP^1) \times \mathrm{Bun}_2^0(\PP^1)$
\end{enumerate}
If on the other hand such $A(N)$ do not exist, then there exists an effective divisor $D$ on $\PP^1$ and an infinite subsequence such that
$L_i \cong \oO_{C_i}(D_i)$ and $\pi_*(D_i)=D$. In this case, the pushforward measures for this subsequence converge to $\mu_D$ defined in Appendix \ref{smallshift}
\end{theorem}
\begin{proof}
Without loss of generality,
we replace $\Mm_i$ with a trivial line bundle and $\Mm_i'$ with $\Mm_i^{-1} \otimes \Mm_i'$.  We also change the map to
$$\Pi: \Ll \mapsto  \left( \pi_*\Ll , \pi_*(\Mm_i \otimes \Ll^{-1}) \right)$$
precomposing with inversion in the second factor does not change the pushforward
measure.  Checking convergence of measure on
e.g. $2\N \times (k + 2\N)$ or $(2 \N + 1) \times (k + 2 \N)$ reduces to Theorem \ref{thm:theta}.
For $a, b \ge 1$,
$$\mu_i( (a + 1 + 2 \N) \times (b + 1 + 2 \N) ) =  \# \left(\Theta_{g-a} \cap \Mm_i - \Theta_{g-b} \right)(\F_q)$$
This set is empty unless $\deg \Mm_i + (g-a) + (g-b) \cong 0 \pmod 2$; note in particular
the pushforward measure is supported on either the set (0) or (1) in the statement of the theorem according to
the parity of $\Mm_i$.  Thus to get convergence we must now pass to a subsequence of line bundles with fixed parity.

We now set up our basic dichotomy.  Fix $0 < \epsilon < 1$.

\vspace{2mm} {\bf Case 1.} There exists a subsequence $(C_i, \Mm_i)$ such that $\Mm_i \in \Theta_{\epsilon g(C_i)}$
for all $i$.  We pass to this subsequence.

The desired statement, including the claim about the measures $\mu_D$, is the exact
analogue  in the function field case of the Corollary in [\cite{EMiV}, \S 10.3] and the ensuing remark, and the proof carries over.
In the language of that corollary (taking $\delta = 1-\epsilon$), $\Q(\sqrt{-d})$ is the analogue of $C_i$,
$\mathfrak{p}_d$ is the analogue of $L_i$, $N(\mathfrak{p}_d)$ is the analogue of the largest $N$ such that $L_i\notin \Theta_N$, and $\mathrm{SO}_3(\Z)\backslash S^2$ is the analogue of
$\mathrm{Bun}_2$.  The proof is written in the same adelic language as in Appendix \ref{sec:adeles}.
They study ergodic theory of the imbedded torus $T(\oO_v)$ for $v$ a place that splits.  In the number field
setting the existence of such $v$ necessitates the assumption that $d\equiv \pm1 (5)$.
In our setting we require no such assumption: since the Weil conjectures guarantee that some point $v \in \PP^1$
of degree at most $2\log(g_i)$ splits in $C_i$, and this suffices for the proof \cite{ELPC}.

\begin{remark}
We sketch a different approach to the same result, using more analytic methods.   It would first
be necessary to develop a suitable analogue of Waldspurger's formula  \cite{W} in the function field setting
(such a formula is developed in \cite{AT} but not in sufficient generality). One breaks up the map into
$$\Pic(C)/\pi^*\Pic(\PP^1)\xrightarrow{F}X_0(D)\rightarrow \mathrm{Bun}_2\times\mathrm{Bun}_2.$$
Now, letting $d=\deg(D)$ one shows that there are $q^{d+o(d)}$ orthonormal Hecke cusp-forms $\phi_i$ on $X_0(D)$ and for each one of these $||\phi_i||_{\infty}=q^{-\frac{d}{2}+o(d)}$. Using Waldspurger's formula, one shows

$$\int_{Pic(C)/\pi^*Pic(\PP^1)}F^*\phi_i = L(\frac{1}{2},\phi_i\otimes\chi_{C})=o(q^{-\frac{d}{2}-\frac{g_C}{2}+o(g_C+d)})$$ where the last inequality follows by work of Deligne.
 Now, if we keep $D$ fixed and vary $C_i$ this implies that the pushfoward measures under $F$ converge to $\mu_{X_0(D)}$, which implies the claim. If $D$ varies but with
 $d<\epsilon g$ then one deduces the result by taking a test function $f$ on $\mathrm{Bun}_2\times\mathrm{Bun}_2$ with $\int f = 0$, and noticing that the pullback (ignoring Eisenstein series for simplicity)
 is a sum $\sum_i c_i\phi_i+c_{f,D}$ with $\sum_{i}|c_i|=O(q^d)$, and thus
 $$\int_{F_*Pic(C)/\pi^*Pic(\PP^1)}f = c_{f,D}+ q^{\frac{d-g_C}{2} + o(d+g_C)}=c_{f,D}+o(q^{-\frac{\delta g_C}{2}+o(g_C)}).$$

The claim now follows from the fact that $c_{f,D}\rightarrow 0$, which is equivalent to the fact that $\mu_D$ tends to the natural measure as $\deg D\rightarrow\infty$, which follows from Deligne's proof of the Ramanujan conjectures.
\end{remark}

\vspace{2mm} {\bf Case 2.}  There exists a subsequence $(C_i, \Mm_i)$ such that $\Mm_i \notin \Theta_{\epsilon g(C_i)}$
for all $i$.  We pass to this subsequence and turn to geometry.

As $\Mm_i$ is anyway defined
only up to a multiple of $\kappa$ we now fix this multiple and take $\deg \Mm_i = 2g  - a - b$.  The statements regarding
the measure when one of $a, b$ is $0$ or $1$ reduce to Theorem \ref{thm:theta}; it remains to treat
the case when $a, b \ge 2$.

For clarity, let us assume $a, b, \deg \Mm$ are all even; the remaining cases differ only notationally.  Then we must show
$$
\frac{\mu_i( (a + 2 \N) \times (b + 2 \N) ) }{\mu_i(\mathrm{Bun}_2^0(\PP^1) \times \mathrm{Bun}_2^0(\PP^1))}
\to \frac{\mu( (a + 2 \N) \times (b + 2 \N) ) }{\mu(\mathrm{Bun}_2^0(\PP^1) \times \mathrm{Bun}_2^0(\PP^1))}
$$
or in other words
$$
\frac{  \# \left(\Theta_{g-a+1} \cap \Mm_i - \Theta_{g-b+1} \right)(\F_q)}{  \# (J_0(C))(\F_q)} \to q^{2-a-b}
$$

According to Theorem \ref{cohomology}, for any $L$ with $e(L) \le g - a - b$, we have
$$\HC^i(\Theta_{g-a} \cap L-\Theta_{g-b}, \Qlbar) \cong \HC^{i+2a + 2b}(J,\Qlbar)(a+b) \,\,\,\,\,\,\,\,\, \mathrm{for}\,\,\,\,  i > g - a - b + e(L)$$
and on the other hand we have
$\dim \HC^i(\Theta_{g-a} \cap L-\Theta_{g-b}, \Qlbar) < N^g$.
Thus we have

 \begin{align*} \frac{\# (\Theta_{g-a} \cap L-\Theta_{g-b} ) (\F_q)}{\#J(\F_q)}
 &=\frac{\sum_{i=0}^{2(g-a-b)}(-1)^i \trf \HC^i(\Theta_{g-a} \cap L-\Theta_{g-b},\Qlbar)}{\sum_{i=0}^{2g}(-1)^i \trf \HC^i(J,\Qlbar)}\\
&=\frac{\sum_{i=g-a-b+e(L)}^{2(g-a-b)}(-1)^i \trf \HC^i(\Theta_{g-a} \cap L-\Theta_{g-b},\Qlbar)+O(N^g\cdot q^{\frac{g-a-b+e(L)}{2}})}
{\sum_{i=g+ a + b + e(L)}^{2g}(-1)^i \trf \HC^i(J,\Qlbar)+O(N^g\cdot q^{\frac{g + a + b + e(L)}{2}})}\\
&=\frac{\sum_{i=g+a+b+e(L)}^{2g}(-1)^i \trf \HC^{i}(J,\Qlbar)(a+b) +O(N^g\cdot q^{\frac{g-a-b+e(L)}{2}})}
{\sum_{i=g+a + b +e(L)}^{2g}(-1)^i \trf \HC^i(J,\Qlbar)+O(N^g\cdot q^{\frac{g + a + b + e(L)}{2}})}\\
&=q^{-a-b}+O(N^g\cdot q^{\frac{e(L) -a -b - g}{2}}). \end{align*}

Recall $\Theta_{k-2E} = \{L \in J^{k}: e(L) \ge E\}$, i.e. $L \notin \Theta_r \iff e(L) < (k-r)/2$.  For us
$k = 2g-a-b$, so we have $L \notin \Theta_r \iff e(L) < g - (a+b+r)/2$, and so the error term for
$L \notin \Theta_r$ is $O(N^g q^{- r/4})$.  (We are studying the evaluation of the limit measure
on some fixed $a,b$, which thus behave as constants when $g \to \infty$.)
Thus since
$\Mm_i \notin \Theta_{\epsilon g}$, the error term becomes $O(N^g q^{-\epsilon g/4})$, which tends to zero
as $g \to \infty$ so long as $q > N^{4/\epsilon}$. Taking $\epsilon\to 1$ establishes the claim.
\end{proof}

\appendix

\section{Adelic generalities} \label{sec:adeles}

We recall here the setting of the equidistribution conjecture.   Let $\A_k$ denote the adeles over
a number or function field $k$.  Consider an algebraic group $G$ over $k$
with a finite volume symmetric space $X_G := G(k) \backslash G(\A_k)$; we write
$\mu_G$ for its normalized Haar measure.  Let $H \subset G$ be a subgroup such that $X_H$ has finite volume, which we normalize to $1$.
For $g \in G(\A_k)$,
there is a map $\rho_{g}: X_H \to X_G$ given by $h \mapsto \rho(h)g$ and a pushforward measure $\rho_{g*} \mu_H$.
In this setting,
when $G$ is a connected group with simply connected cover $\tilde{G}$,
we expect the set of measures  $\{ \rho_{g*}^0 \mu_H\} \cup \{0\}$ to be weak$-*$ closed.

When $k$ is a function field, such statements can be rephrased in terms of vector bundles over curves; explaining
this in detail for $G = \PGL_2$ is the task of this Appendix.

\subsection{Double coset spaces as moduli of bundles}

Let $K \subset G(\A_k)$ be a compact open subgroup.  We may pose the equidistribution conjecture for the space
$X_G / K$.  Moreover, as $X_G = \lim_K X_G /K$, the equidistribution conjecture holds for $X_G$ iff it holds for
all such quotients, and we may moreover restrict ourselves to $K$ of the form $K = \prod K_\nu$ over the places $\nu$ of $k$.
Likewise, given a subgroup $H < G$, we may descend the map $\rho_g:X_H \to X_G$ to some
$X_H/K_H \to X_G / K_G$ for any $K_H \subset H\cap g K_G g^{-1}$.

Henceforth we take $k = \F_q(C)$, the function field of a curve $C$.  Give a pair $(G,K)$ of a group $G$ over $k$ and an open compact
subgroup $K=\prod_v K_v\subset G(\A_k)$ we may define a sheaf of groups $\mathfrak{G}$ by $$\mathfrak{G}(U):=G(k)\cap\prod_{v\in U}K_v.$$

If $G$ is defined over $\F_q$,
we may consider the maximal compact subgroup $G(\oO_C) = \prod G(\oO_\nu)$;
the quotient is the space of Zariski locally trivial $G$-bundles over $C$:
$$G(k) \backslash G(\A_k) / G(\oO_C) = \mathrm{Bun}_G (C)$$
Explicitly, the identification is the following: to an element $g \in G(\A_k)$, we assign the bundle $E_g$ whose sections
are given by
$$E_g(U) = G(k) \cap \prod_{\nu \in U}  G(\oO_\nu)g_\nu^{-1}. $$

For the special case where $G=\GL_r$, we can identify $X_G$ with $\mathrm{Bun}_r$ as follows: using the natural action of $\GL_r(k)$ on $k^{\oplus r}$, we define a vector bundle $V_g$ by
\begin{equation}\label{bunvec}
V_g(U) = k^{\oplus r}\cap \prod_{\nu\in U} g\oO_{\nu}^{\oplus r}.
\end{equation}

\vspace{2mm}

\subsection{Rank 1 Tori}
We explain here how in the case of rank 1 tori in $\PGL_2$, the map $\rho_g$ amounts in the simplest case to the pushforward map along a double cover of curves from line bundles to rank 2 vector bundles.

By descent, $\mathrm{Hom}\left(\mathrm{Gal}(\overline{k}/k),  \mathrm{Aut}_{\overline{k}} (\G_m) \cong \mathrm{GL}_1(\Z) \right)$
classifies rank 1 tori over $k$.
Explicitly, given
$\rho \in \mathrm{Hom} (\mathrm{Gal}(\overline{k}/k),  \pm 1 )$, the fixed field of $\ker \rho$ is a degree two extension $l$ of $k$, i.e.
the function field of a degree two cover  $ \pi: D \to C$. The corresponding torus is $T \cong_k (\Res_{l/k} {\G_m})/ \G_m$.
The map
$\Res_{l/k} \G_m \to \Aut\,\, \Res_{l/k} \G_a \cong {\GL_2}_{/k}$
gives rise to a conjugacy class of inclusions $T \hookrightarrow \PGL_2$.
On $k$ points, we are just saying $T(k) = l^*/k^*$ and the map
$T \to \PGL_2$ comes from inclusion of $l^*$ into the $k$-vector space automorphisms of $l$, together
with a choice of isomorphism of $k$-vector spaces $l\cong k^{\oplus 2}$.
One can show that every non-split torus of ${\PGL_2}_{/k}$ arises in this manner.

There is also a canonical map from $\mathrm{Pic}(D)/\pi^* \mathrm{Pic}(C)$ to $\mathrm{Bun}_{\PGL_2}(C)$ given by pushforward.
Its adelic description involves the {\em choice} of isomorphisms $l\cong k^{\oplus 2}$.  Indeed,
 letting $\oO_{l,v}$ denote the integral closure in $l$ of $\oO_v$, we have two lattices
 $$\oO_{l,v} \subset l_v \cong k_v^{\oplus 2} \supset \oO_v^{\oplus 2}$$
Let $\gamma_v \in \GL_2(k_v) $ be any element such that $\oO_{l,v} = \gamma_v \cdot \oO_v^{\oplus 2}$, and let
$\gamma = \prod_v \gamma_v$.
Consider any $t\in \GL_1(\A_l)$ and let $V_t$ denote  the corresponding line bundle.
It follows tautologically from Equation \eqref{bunvec} that
 $\pi_*(V_t) \cong V_{(\Res_{l/k} t) \gamma}$. Note also
$\oO_{l,v}^{\times} \subset \gamma_v\GL_2(\oO_v)\gamma_v^{-1}$.

Since $T$ is a torus there is a unique maximal compact subgroup $T(\oO) < T(\A_k)$.  The simplest maps $\rho_g:X_T \hookrightarrow X_{\PGL_2} $
of symmetric spaces, and the only ones we consider in this paper,  are those which descend to the final quotient:
$$\mathrm{Pic}(D) / \pi^* \mathrm{Pic}(C) = X_T/T(\oO) \to X_{\PGL_2} / \PGL_2 (\oO) = \mathrm{Bun}_{\PGL_2}(C)$$
That is, we want
$g^{-1}T(\oO)g \subset\PGL_2(\oO)$. The $\gamma$ constructed above is such an element.  Moreover, the set of such $g$ forms
a single double coset in $T(\A_k) \backslash \PGL_2(\A_k) / \PGL_2(\oO)$, as can be seen by doing a local calculation.

In other words, any two such maps $X_T \to X_{\PGL_2}$ differ by pre-composing with multiplication by some element of $X_T/T(\oO)$,
and hence all such maps take the form $M \mapsto \pi_* (M \otimes L)$ for some fixed line bundle $L$.

\subsection{Completely framed vector bundles}  We briefly describe in more geometric terms  the
 ``pushforward'' map $\rho_\gamma:X_T\hookrightarrow X_{\GL_2}$ induced by $\gamma$.

By a { \em completely framed} vector bundle, we mean a vector bundle $V$ together with
a frame $S_v: \oO_v^{\oplus k} \cong V_v$ of the completed stalk $V_v$ at each place $v$.
An isomorphism of completely framed vector bundles is just an isomorphism of the underlying vector
bundles which preserves the frames.

The space $X_{\GL_r} = \GL_r(k) \backslash \GL_r(\A)$ parameterizes isomorphism
classes of completely framed vector bundles.
Indeed, to $g\in\GL_r(\A_k)$, we associate the vector bundle
$V_g$ of Equation (\ref{bunvec}), together with the $\oO_v$-frame given by $\oO_v^{\oplus k} \xrightarrow{\cdot g_v} V_{g,v}$ of the completed stalk $V_{g,v}$
at every place $v$ of $k$.  If $g_0\in\GL_r(k)$, then
$k^r\xrightarrow{\cdot g_0} k^r$ induces an isomorphism $V_g\rightarrow V_{g_0g}$ preserving these frames; on the other hand,
isomorphisms of vector bundle
induce $k$-linear isomorphisms of their  meromorphic sections.

Likewise, for a nonsplit rank one torus $T$ coming from the degree two field extension $l/k$ or equivalently the degree two cover $\pi: D \to C$,
the space $X_T$ parameterizes isomorphism classes
of: a line bundle $\Ll$ on $D$, together with, for each place {\em of $k$}, a framing of the completed stalk $\oO_{l, v} \to \Ll|_{\pi^{-1}(v)}$.

In these terms, the ``pushforward'' map  $\rho_{\gamma}:X_T\rightarrow X_{\GL_2}$ sends a line bundle $\Ll$ to $\pi_* \Ll$, and
creates frames by composition:
$$ \oO_v^{\oplus 2} \xrightarrow{\cdot \gamma} \oO_{l,v} \to \Ll|_{\pi^{-1}(v)} = (\pi_* \Ll)_v$$
 This has slightly different geometric behavior at points $v$ of $C$ that split and at points that stay inert or ramify. If $v$ splits into two points $w_1,w_2$ in $D$ then $\mathcal{L}$ has distinguished vectors $\ell_{w_1}\in\mathcal{L}_{w_1}$
and $\ell_{w_2}\in\mathcal{L}_{w_2}$, giving rise to a  basis of $\pi_*\mathcal{L}_v$.
However, if $v$ stays inert or ramifies then $\mathcal{L}_v$ only has the one distinguished vector $\ell_v$ and we must use the ring structure of $\oO_{l,v}$ to get a second basis element for $\pi_*\mathcal{L}_v$ over $\oO_v$.

\subsection{Hecke Measures} \label{smallshift}
Consider the diagonal subgroup $\PGL_2(\A) \subset \PGL_2(\A) \times \PGL_2(\A)$.
Fix an effective divisor  $D=\sum_v n_vP_v$ on a curve $C$.  Let $\pi_v$ be a uniformizer of $\oO_v$, and consider
the element
$$g_D = \left(1, \prod_{v\in C} \left(\begin{smallmatrix} 1 & 0 \\ 0 & \pi^{n_v}\end{smallmatrix}\right)\right)  \in \PGL_2(\A)  \times \PGL_2(\A)$$
Let $\mu_D$ be the pushforward of the natural measure on
$X_{\PGL_2}$ to $\mathrm{Bun}_2\times\mathrm{Bun}_2$ along
$$X_{\PGL_2} \xrightarrow{\rho_{g_D}} X_{\PGL_2} \times X_{\PGL_2} \to \Bun_2 \times \Bun_2$$.

The composition factors through
 $X_0(D):=X_{\PGL_2}/K_D$, where  $K_D\subset\PGL_2(\oO)$ is the compact subgroup of elements
$$\begin{pmatrix} a_v& b_v\\ c_v & d_v\end{pmatrix}_{v\in C}\textrm{ such that } \forall v\in C, \pi_v^{n_v}|c_v .$$

Equation \eqref{bunvec} gives a geometric description of $X_0(D)$:
it parameterizes rank 2 vector bundles $V$ equipped with a flag $\mathfrak{F}_v:\oO_v/\pi_v^{n_v}\oO_v \hookrightarrow  V_v/\pi_v^{n_v}V_v$ for each point $v$,
up to twisting by line bundles; that is, $(V,\mathfrak{F}_v)\cong (V\otimes\mathcal{L}, \mathfrak{F}_v\otimes \mathcal{L}_v)$.
Equivalently, this can be described as a pair of vector bundles $W\subset V$ such that $V/W$ is isomorphic to $\oO_D$, up to the equivalence
$(V,W)\cong (V\otimes\mathcal{L},W\otimes{L})$.
The map to  $\mathrm{Bun}_2 \times \mathrm{Bun}_2$ just takes such an object to $(V, W)$.

\end{document}